\documentclass[10pt]{amsart}
\usepackage{amscd,amssymb,graphics}

\usepackage{amsfonts}
\usepackage{amsmath}
\usepackage{amsxtra}
\usepackage{latexsym}
\usepackage[mathcal]{eucal}

\usepackage{graphics,colortbl}

\input xy
\xyoption{all}
\usepackage{epsfig}

\usepackage[pdftex,bookmarks,colorlinks,breaklinks]{hyperref}

\newtheorem{theorem}{Theorem}[section]
\newtheorem{fact}[theorem]{Fact}
\newtheorem{corollary}[theorem]{Corollary}
\newtheorem{lemma}[theorem]{Lemma}
\newtheorem{proposition}[theorem]{Proposition}

\newtheorem{question}[theorem]{Question}
\newtheorem{definition}[theorem]{Definition}
\numberwithin{equation}{section}
 
\newtheorem{prob}[theorem]{Problem} 

\theoremstyle{remark}
\newtheorem{remark}[theorem]{Remark}
\newtheorem{example}[theorem]{Example}

\newcommand{\ben}{\begin{enumerate}}
\newcommand{\een}{\end{enumerate}}
\newcommand{\bit}{\begin{itemize}}
\newcommand{\eit}{\end{itemize}}

\def\R {{\mathbb R}}
\def\Q {{\mathbb Q}}
\def \F {{\mathbb F}}

\def\N{{\mathbb N}}
\def\T{{\mathbb T}}

\def\Z {{\mathbb Z}}

\def\Homeo{{\mathrm{Homeo}}\,}

\def\QED{\nobreak\quad\ifmmode\roman{Q.E.D.}\else{\rm Q.E.D.}\fi}

\def\s {\sigma}
\def\g {\gamma}
\def\t {\tau}
\def\sk {\vskip 0.3cm}

\def\GL{\operatorname{GL}}

\def\UT{\operatorname{UI}}
\def\UT{\operatorname{UT}}
\def\Aut{\operatorname{Aut}}

\begin{document}

\title[]{Key subgroups in the Polish group of all automorphisms of the rational circle} 

		\author[]{Michael Megrelishvili}
		\address[M. Megrelishvili]
	{\hfill\break Department of Mathematics
		\hfill\break
		Bar-Ilan University, 52900 Ramat-Gan	
		\hfill\break
		Israel}
		\email{megereli@math.biu.ac.il}


 

 \dedicatory{Dedicated to Vladimir Pestov on the occasion of his 70th birthday}
 \thanks{Supported by the Gelbart Research Institute at the Department of Mathematics, Bar-Ilan  University}

 \begin{abstract}
 	Extending some results of a joint work with E.  Glasner, we continue to study the Polish group
 	$G:=\Aut(\mathbb Q_0)$ of all circular order preserving permutations
 	of the rational circle $\mathbb Q_0=\mathbb Q/\mathbb Z$, endowed with the pointwise
 	topology. 
 	We show that the point stabilizers $H=G_q$ are extremely amenable \textit{inj-key} subgroups of $G$ (that is, they distinguish coarser Hausdorff group topologies on $G$), but are not co-minimal
 	in $G$).  
 	
 	These examples answer a question posed in a 
 	joint work with M. Shlossberg and are inspired by a question of V. Pestov concerning Polish groups with metrizable universal
 	minimal flow. It remains an open problem to study Pestov's question in its full generality.
 \end{abstract}
\maketitle
	

\section{Introduction}  

The Polish group $G:=\Aut (\Q_0)$ of all circular order preserving permutations of $\Q_0$ has several remarkable dynamical 
properties. Recall that by results of \cite{GM-UltraHom} its universal minimal $G$-flow $M(G)$ is dynamically small, being a circularly ordered metrizable compactum. Hence, $M(G)$ is a tame dynamical system, \cite{GM-CircOrd18,GM-TC}.   
One of the motivations of \cite{GM-UltraHom} was to study generalized amenability of topological groups. 
 
 Geometrically, $M(G)$ can be obtained by splitting all rational points of $\Q_0=\Q/\Z$ on the circle $\T=\R/\Z$.  This result is, in some sense, a ``circular analog" of a well-known seminal result of V. Pestov \cite{Pest98} about extreme amenability of the Polish group $\Aut(\Q,\leq)$ of all linear order preserving permutations of $\Q$. 

Note that there are several additional cases 
where $M(G)$ is a circularly ordered $G$-flow. 
For example, according to a result of L. Nguyen van Th\'{e} in \cite{van-the}, this happens for 
the	Polish groups $\Aut(\mathbf{S}(2))$ and 
$\Aut(\mathbf{S}(3))$ of automorphisms 
of the circular directed graphs $\mathbf{S}(2)$ and $\mathbf{S}(3)$.  

In the present article, we study some new properties (both dynamical and topological) of $\Aut (\Q_0)$. 
One of our goals is to resolve a question posed in a recent joint work with M. Shlossberg \cite{MeSh-key} concerning natural examples of inj-key subgroups that are not co-minimal (Problem \ref{p:main} below).

A subgroup $H$ of a Hausdorff topological group $(G,\g)$ is said to be \textit{inj-key} (see Definition \ref{d:inj-key}) if $H$ distinguishes 
weaker Hausdorff group topologies $\s \subseteq \g$ on $G$. 
That is, when the restriction map  
$$r_H \colon \mathcal{T}_{\downarrow}(G) \to \mathcal{T}_{\downarrow}(H), \ \s \mapsto \s|_H$$ is \textbf{injective}, where $\mathcal{T}_{\downarrow}(G)$ is the sup-semilattice of all coarser Hausdorff group topologies on $G$. 
Recall also the notion of co-minimality introduced in a joint work with D. Dikranjan \cite{DM10}. $H$ is \textit{co-minimal} (Definition \ref{d:co-min}) in $G$ means that 
for every $\s \in \mathcal{T}_{\downarrow}(G)$ the coset topology $\s/H$ on $G/H$ is just the original coset topology $\g/H$. 

According to results of \cite{DM10} and \cite{MeSh-key}, many remarkable subgroups are co-minimal. 
For example, the center $H:=Z(G)$ of $G$ is co-minimal for 
$G:=\UT (n,\F)$, the matrix group of all unitriangular matrices, 
and also for the general linear group $G:=\GL(n,\F)$, 
 where $\F$ is a local field.  

Every co-minimal subgroup is inj-key (Fact \ref{p:inj-key=co-min for central}.1). 
In the present work   
we give examples of inj-key subgroups which are not co-minimal. 
More precisely, we show that for every \(q\in\Q_0\), the stabilizer
\(G_q\) is an extremely amenable inj-key subgroup of \(\Aut(\Q_0)\),
but is not co-minimal in \(\Aut(\Q_0)\) (Theorem  \ref{t:counterexample}).  
This justifies the concept of \textit{inj-key subgroups} and answers the following general question posed in \cite{MeSh-key}.

\begin{prob} \label{p:main}  \cite{MeSh-key}  
	Find examples of inj-key subgroups which are not  co-minimal. 
\end{prob}

 Looking for a suitable example, we followed 
V. Pestov's suggestion to examine Polish groups $G$ with \textbf{metrizable} \textit{universal minimal $G$-flow} $M(G)$. 
 The following question was suggested 
by V. Pestov during the online conference, \textit{Algebra, Topology and Their Interactions}, July 2024.  
\begin{question} \label{q:Pest} (Pestov) 
	Let $G$ be a Polish group with a (extremely amenable) subgroup $H$ such that the universal minimal $G$-flow $M(G)$ is a  metrizable completion of $G/H$.
	 Under which conditions is $H$ inj-key in $G$? When is $H$  inj-key but not co-minimal?
\end{question}  

   
Recall that 
	Polish groups with metrizable $M(G)$ are very important in several 
	research trends (for instance, in abstract topological dynamics; see \cite{KPT,BMT,Zucker14,BZ}.  
	By a well-known result \cite{BMT},   
	for such $G$ there exists a closed subgroup $H$ such that $M(G)$ is the completion of the coset $G$-space $G/H$. By another result \cite[Prop. 3.3]{MNT}, $H$ must be extremely amenable.     
	
 Such topological groups frequently are minimal. 
 See, for example:  
 \begin{enumerate}
 	\item Glasner--Weiss \cite{GW-02} with $G=S_{\N}$ the symmetric group. 
 	\item Glasner--Weiss \cite{GW-03}, $G=H(2^{\N})$  homeomorphism group of the Cantor set. 
 	\item Pestov \cite{Pe-nbook,Pest98}, $G=\Homeo_+(\T)$.
 	\item By results of Kwiatkowska \cite{Kwi18} and 
 	Duchesne \cite{Duch20} the homeomorphism groups $G:=H(W_S)$ of Wazewski dendrites (for finite $S$) have metrizable $M(G)$. This group is minimal (moreover, its topology is the minimum Hausdorff group topology) by \cite[Theorem 1.7]{Duch20}.  
 \end{enumerate}  
 
In minimal groups, every subgroup is co-minimal. 
Thus, this case is too restrictive to provide an answer to Problem \ref{p:main} and Question \ref{q:Pest}. 
 However, groups with metrizable $M(G)$, perhaps are close to being minimal as the present example of $\Aut (\Q_0)$ hints. Here we treat the presence of an inj-key but not co-minimal subgroup as a kind of minimality property.  

Theorem~\ref{t:counterexample} gives a concrete response to Problem~\ref{p:main}. The example was found
by examining the geometry of the metrizable universal minimal flow of
\(\Aut(\mathbb Q_0)\), following the direction suggested in Question~\ref{q:Pest}. 
The role of Question~\ref{q:Pest} in the present paper is motivational: it leads to the
study of stabilizers arising from the natural compactifications of
\(\Aut(\mathbb Q_0)\). The rational stabilizers provide the desired inj-key
but non-co-minimal examples.

One of the important ingredients in the proof is a theorem of Chang--Gartside
\cite[Theorem 8]{CG} about minimum Hausdorff group topologies.
We also use, as in \cite{GM-UltraHom}, the circular ultrahomogeneity
of the relevant actions and the geometric description of certain greatest
\(G\)-compactifications via inverse limits of finite circularly ordered sets.
For instance, if \(G=\operatorname{Aut}(\mathbb Q_0)\) and \(H=\operatorname{St}(q_0)\),
then the greatest \(G\)-compactification \(\beta_G(G/H)\) is a circularly ordered
metrizable compact space obtained from the circle \(\mathbb T\) by replacing
each rational point \(q\in\mathbb Q_0\) by an ordered triple
\(q^-,q,q^+\).

It is an open problem to study Pestov's question in its full generality. We believe that, in general, there are many interesting cases of key (inj-key) subgroups in remarkable geometrically defined topological groups (not necessarily responding to Question \ref{q:Pest}). See a discussion and open questions in Section \ref{s:Q}.   



	\vskip 0.3 cm
\noindent \textbf{Acknowledgment.}  
Many thanks to Eli Glasner and Menachem Shlossberg for the inspiration while writing joint works \cite{GM-UltraHom} and  \cite{MeSh-key} respectively. 
I am grateful to Vladimir Pestov for proposing to examine topological groups with metrizable universal minimal flow as a possible counterexample (Question \ref{q:Pest}).

\sk 
\section{Some definitions and useful properties}   

In what follows, all topological spaces are assumed to be Tychonoff. In particular, all topological groups are Hausdorff. 
For every topological group $(G,\g)$ and its subgroup $H,$ denote by $G/H:=\{gH: g \in G\}$ the \textit{coset $G$-space} endowed with the usual quotient topology $\g/H$ and the natural continuous action $G \times G/H \to G/H$.  Denote by $N_e(G)$ (or, $N_e(G,\g)$) the set of all $\g$-neighborhoods of the neutral element $e \in G$.   

The following system of entourages 
$$
\widetilde{U}:=\{(xH,yH) \in G/H \times G/H: \ \ 
x \in UyH \} \ \ \
(U \in N_e(G))
$$
is a base of a topologically compatible uniformity $\mathcal{U}_r(G/H)$ on the coset space
$G/H$, the so-called \emph{right uniformity}, \cite[Theorem 5.21]{RD}. 
The projection $q\colon G \to G/H$ is a continuous open $G$-map, uniform with respect to the right uniformities $\mathcal{U}_r(G)$ and $\mathcal{U}_r(G/H)$. 
\sk   
For a topological group $(G,\g)$, denote by $\mathcal{T}_{\downarrow}(G,\g)$ (or, simply $\mathcal{T}_{\downarrow}(G)$) the poset (in fact, sup-semilattice) of all coarser Hausdorff group topologies $\t \subseteq \g$ on $(G,\g)$.   
The extreme case of the singleton, where $\mathcal{T}_{\downarrow}(G,\gamma)=\{\gamma\}$, is equivalent to the minimality of $(G,\gamma)$.

 Minimal topological groups are in widespread use in mathematics. 
 There is an extensive literature about them.  
See for example, the survey paper \cite{DM14} (and references therein) and also \cite{DPS89,DM10, MS-Fermat,MeSh-key}. 
Recall the following useful relaxations of the minimality condition for subgroups.


\begin{definition} \label{d:co-min} 
Recall that a subgroup \(H\) of \(G\) is said to be \emph{relatively minimal}
\cite{MEG04} (respectively, \emph{co-minimal} \cite{DM10}) in \(G\) if every
coarser Hausdorff group topology
\(\sigma\in\mathcal T_{\downarrow}(G,\gamma)\) induces on \(H\) the original
topology, that is, \(\sigma|_H=\gamma|_H\) (respectively, the quotient topology
\(\sigma/H\) on \(G/H\) coincides with \(\gamma/H\)).
\end{definition}



Clearly, in minimal groups any subgroup is both relatively minimal and co-minimal. 
 A dense subgroup is always co-minimal by \cite[Lemma 3.5.3]{DM10}.

\sk  
In Definition \ref{d:inj-key}, we recall two additional minimality conditions introduced very recently in a joint work with M. Shlossberg \cite{MeSh-key}.

For every pair $\s, \t \in \mathcal{T}_{\downarrow}(G,\g),$ we have $\sup\{\s,\t\} \in \mathcal{T}_{\downarrow}(G,\g)$. 
The following family 
\begin{equation} \label{eq:sup-top} 
	\{U \cap V: \ \ U \in N_e(G, \s), V \in N_e(G, \t)\}
\end{equation} 
constitutes a local base of $\sup\{\s,\t\}$ 
at the neutral element $e \in G$.   
If $H$ is a subgroup of $G$, then we have a natural order preserving restriction map 
$$r_H \colon \mathcal{T}_{\downarrow}(G) \to \mathcal{T}_{\downarrow}(H), \ \s \mapsto \s|_H.$$  
Furthermore, it is a morphism of sup-semilattices because   
for every pair $\s,\t \in \mathcal{T}_{\downarrow}(G)$ we have (using  Equation \ref{eq:sup-top}) 
$
\sup\{r_H(\s),r_H(\t)\}=r_H(\sup\{\s,\t\}). 
$

\begin{definition} \label{d:inj-key} \cite{MeSh-key} 
	Let $H$ be a subgroup of a topological group $(G,\gamma)$.  
	\begin{enumerate}
		\item $H$ is a 
		\textbf{key subgroup} of $G$ if for every 
		$\s \in \mathcal{T}_{\downarrow}(G)$ 
		 the coincidence $\s|_H = \gamma|_H$ 
		implies that  $\s = \gamma$. 
		\item $H$ is 	
		an \textbf{inj-key subgroup} of $G$ if for every pair 
		$\s_1,\s_2 \in \mathcal{T}_{\downarrow}(G)$ the coincidence $\s_1|_H = \s_2|_H$ 
		implies that  $\s_1 = \s_2$. Equivalently, this means that $$r_H \colon \mathcal{T}_{\downarrow}(G) \to \mathcal{T}_{\downarrow}(H), \ \s \mapsto \s|_H$$ is \textbf{injective}.  
	\end{enumerate}  
\end{definition}

We have the implications   

\centerline{\textbf{co-minimal $\Rightarrow$ inj-key $\Rightarrow$ key.}}  

\noindent 
If $G$ is minimal then every subgroup is co-minimal in $G$. 
The concept of key subgroups appears implicitly in some earlier publications \cite{MEG95, DM10, MS-Fermat}. Recall some  results from \cite{MeSh-key}.  
If $H$ is a central subgroup, then $r_H$ is onto and so 
the injectivity of $r_H$ can be replaced by bijectivity and we get an isomorphism of semilattices.  In general, $r_H$ need not be onto (even for abelian subgroups of index 2).


Recall that a closed subgroup $H$ of $G$ is said to be \textit{co-compact} if  the coset $G$-space $G/H$ is compact.  
In general, $r_H$ need not be injective even when $H$ is a  key subgroup. For example, a co-compact subgroup $\Z^n$ is key in $\R^n$ but not inj-key 
(see \cite{MeSh-key}).   

If $G$ is locally compact abelian, then its subgroup $H$ is inj-key (equivalently, co-minimal) if and only if $H$ is co-compact. There exists a discrete abelian countable infinite (non-minimal) group $G$ containing a proper co-minimal subgroup $H$. 

A topological group with complete two-sided uniformity ${\mathcal U}_{l \vee r}$ is said to be \textit{Raikov complete}. Every locally compact or Polish group is Raikov complete. 
A closed subgroup $H$ of $(G,\g)$ is said to be \textit{strongly closed} (in the sense of \cite{DM10}) if $H$ remains closed in $(G,\t)$ for every  $\t \in \mathcal{T}_{\downarrow}(G,\g)$.


%


\begin{fact} \label{p:inj-key=co-min for central} \cite{MeSh-key} 
	\begin{enumerate}
		\item Every co-minimal subgroup is inj-key.  
		\item Let $H$ be a central subgroup of $(G,\g)$. Then $H$ is inj-key iff $H$ is co-minimal in $G$. 	 
\item Let $H$ be a co-compact Raikov complete subgroup in $G$.   Then $H$ is a key subgroup of $G$. 
\item Let $H$ be a strongly closed subgroup in $G$  which is co-compact and Raikov complete. Then $H$ is co-minimal.  
	\end{enumerate}	
\end{fact}

	In the assertion (4) compactness of $G/H$ cannot be replaced by precompactness (in the right uniformity) as Theorem \ref{t:counterexample} and Lemma \ref{l:minimum} demonstrate. 

\begin{example} \label{ex:Effros} 
	Let $K$ be a compact homogeneous metrizable space and $G:=\Homeo(K)$ be its group of all homeomorphisms with respect to the compact-open topology. Then $G$ is a Polish group. For every point $x_0 \in K$ consider the stabilizer subgroup $H:=St(x_0)$. Then $H$ is a \textit{key subgroup} of $G$. Indeed, the coset space $G/H$ is compact metrizable being naturally homeomorphic to $K$ (by Effros' Theorem).  
	As $H$ is also Raikov complete, we may apply Fact \ref{p:inj-key=co-min for central}.3. 
\end{example}

According to \cite{MeSh-key}, for many interesting groups $G$ the center $Z(G)$ is co-minimal (equivalently, inj-key). Hence the restriction map 
$r_H \colon \mathcal{T}_{\downarrow}(G) \to \mathcal{T}_{\downarrow}(H)$ for the center $H:=Z(G)$ is bijective (isomorphism of semilattices). In particular, this happens for the generalized Heisenberg group $G=(\R \oplus V^*) \rtimes_{w} V$ modeled on the bilinear duality map $w \colon V \times V^* \to \R$ for every Banach space $V$ (with $Z(G) \simeq \R$).  
Additional examples are:  
$G:=\UT (n,\F)$ the matrix group of all unitriangular matrices where $\F$ be a local field 
(with $Z(G) \simeq \F$, the ``corner one-parameter subgroup"), as well as $G:=\GL(n,\F)$ the general linear group. 

In \cite{MeSh-key}, there are many other important examples of co-minimal \textbf{central} subgroups. However, central subgroups cannot answer (by Fact \ref{p:inj-key=co-min for central}.2) the following problem which is one of the motivations of the present  work. 

\begin{prob} \cite{MeSh-key}    
	Find examples of inj-key subgroups which are not  co-minimal. 
\end{prob}
 
It makes sense to examine subgroups of massive non-commutative non-minimal groups.  
Below, in Theorem \ref{t:counterexample}, we give a  concrete natural example which resolves Problem \ref{p:main}. 
It is expected that many other counterexamples can be constructed. However, the present geometric example has its own interest, suggesting some new questions and perspectives for the groups with metrizable $M(G)$. See Questions \ref{q:Pest} and \ref{q:BPestov}. 

\sk 
\subsection*{Group actions and equivariant compactifications} 
\label{s:DynDef} 

Let $$G \times X \to X, (g,x) \mapsto gx=g(x)$$ be a continuous left action of a topological group $G$ on $X$. Then we say that $X$ is a $G$-\textit{space}. If, in addition, $X$ is compact Hausdorff, then $X$ is a $G$-\textit{flow}. 

A continuous function $f \colon X \to Y$ between $G$-spaces is a $G$-\textit{map} (or,\textit{ equivariant}) if $f(gx)=gf(x)$ for every $g \in G, x \in X$. 
If $Y$ is a $G$-flow and $f$ is dense, then $f$ is a $G$-\textit{compactification} of $X$. 
 If $f$ is a topological embedding, then
the $G$-compactification is said to be \emph{proper}.
For every $G$-space $X$, there exists a greatest $G$-compactification $\beta_G \colon X \to \beta_G (X)$, which is not necessarily proper in general 
(see \cite{Me-opit07} and \cite{Pest-Smirnov}). 
For every $G$-compactification 
$f \colon X \to Y$ there  exists a uniquely defined continuous onto $G$-map ($G$-factor) $q \colon \beta_G(X) \to Y$ such that $f=q \circ \beta_G$. 
This serves as an equivariant generalization of the classical Stone-\v{C}ech compactification. 


\begin{fact} \label{t:G/H}  
		\emph{(J. de Vries \cite{Vr-can75})} For every closed subgroup $H$ of $G$ the coset $G$-space $G/H$ is always $G$-compactifiable.  
	In fact, the greatest proper $G$-compactification $\beta_G (G/H)$ is just the Samuel compactification of the right uniform space $(G/H,\mathcal{U}_r(G/H))$. 
	\end{fact}
	
	 Thus, if $\mathcal{U}_r(G/H)$ is precompact (totally bounded), then $\beta_G (G/H)$ is the $G$-completion of $G/H$.   
For more information about $G$-compactifications see  
\cite{Vr-can75,Vr-Embed77} and 
\cite{Me-MaxEqComp,IbMe,Me-b,Me-opit07}.


\sk 
In the sequel, we use several times the following well known minimality property for the compact-open topology. 
\begin{fact} \label{f:MinProp} 
For every compact space $K$ its homeomorphism group $G:=\Homeo(K)$ is a topological group in the compact-open topology $\t_{co}$ and the tautological action $\Homeo(K) \times K \to K$ is continuous. 
Moreover, $\t_{co} \subseteq \g$ for every group topology $\g$ on 
$\Homeo(K)$ which makes that action continuous. 
\end{fact}

\sk 
Let $(G,\t_0)$ be a topological group and $H$ is its closed subgroup such that the action $G \times G/H \to G/H$ is effective. 
Consider the greatest $G$-compactification $\beta_G (G/H)$ which always is proper (result of de Vries Fact \ref{t:G/H}). 
Consider all possible $G$-compactifications $\nu \colon G/H \to Y$ 
(up to the natural equivalence of compactifications) such that $\nu$ is injective (but not necessarily, embedding) on $G/H$.  
We have the uniquely defined $G$-factor map $\beta_G (G/H) \to Y$ (which is injective on $G/H$) and the 
induced continuous action $\pi_f \colon (G,\t_0) \times Y \to Y$. Denote by $\g_f$ the corresponding compact-open topology on $G$ (inherited from $\Homeo(Y)$). Finally, denote by $\Sigma_H(G)$ the set of all such topologies $\g_f$ on $G$. Each $\g_f$ is a group topology and Hausdorff because $\pi_f$ is effective. By the minimality property of the compact-open topology (Fact \ref{f:MinProp}) we have $\g_f \subseteq \t_0$. Thus,  $\Sigma_H(G) \subseteq \mathcal{T}_{\downarrow}(G)$.  

We cannot guarantee equality in general; however, we can sometimes control all possible corresponding coset $G$-space topologies on $G/H$.

\begin{lemma} \label{l:CosetTopologies} 
	Let $(G,\t_0)$ be a topological group and $H$ is a strongly closed subgroup such that the action $G \times G/H \to G/H$ is effective.   
	Then  
	$$
	\{\g/H: \  \g \in \Sigma_H(G)\} = 
	\{\s/H: \  \s \in \mathcal{T}_{\downarrow}(G,\t_0)\}. 
	$$ 	
\end{lemma}
\begin{proof} 
The inclusion $\subseteq$ is clear because $\Sigma_H(G) \subseteq \mathcal{T}_{\downarrow}(G)$, as mentioned above.	 

To establish the reverse inclusion $\supseteq$, consider any $\s \in \mathcal{T}_{\downarrow}(G,\t_0)$. 
	We are going to check that there exists $\g \in \Sigma_H(G)$ such that $\s/H =\g/H$. 
	
	The quotient topology \(\sigma/H \subseteq \tau_0/H\) is Hausdorff because
	\(H\) is \(\sigma\)-closed, being strongly closed in \((G,\tau_0)\).
	Since \(\sigma/H\subseteq \tau_0/H\), the identity map
	$
	(G/H,\tau_0/H)\longrightarrow (G/H,\sigma/H)
	$
	is continuous and \(G\)-equivariant. Hence the greatest
	\((G,\sigma)\)-compactification \(\beta_G(G/H,\sigma/H)\) is also a
	\((G,\tau_0)\)-compactification of \(G/H\), injective on \(G/H\). Therefore
	the corresponding compact-open topology \(\gamma_f\), coming from the action
	of \(G\) on
	$
	Y:=\beta_G(G/H,\sigma/H),
	$
	belongs to \(\Sigma_H(G)\).
	
	Since the action \((G,\sigma)\times Y\to Y\) is continuous, Fact~\ref{f:MinProp} implies
	\(\gamma_f\subseteq \sigma\). Hence \(\gamma_f/H\subseteq \sigma/H\).
	Conversely, the orbit map
	\[
	(G,\gamma_f)\to (G/H,\sigma/H),\qquad g\mapsto gH,
	\]
	is continuous. Since \(\gamma_f/H\) is the strongest quotient topology on
	\(G/H\) making this map continuous, we get
	\(\sigma/H\subseteq \gamma_f/H\). Thus
	$
	\gamma_f/H=\sigma/H.
	$   
		
\end{proof} 

 \sk 
\subsection*{Some definitions from topological dynamics} 
\label{s:DynDef2} 
A \(G\)-flow \(X\) is said to be minimal if it has no proper nonempty closed
\(G\)-invariant subspace. Equivalently, every \(G\)-orbit is dense in \(X\).
Every \(G\)-flow contains a minimal \(G\)-subflow. Note that for every
topological group \(G\) there exists the universal minimal \(G\)-flow \(M(G)\).
The latter means that every minimal \(G\)-flow is a \(G\)-factor of \(M(G)\).
For more properties of \(M(G)\), we refer to \cite{Gl-book}.  

One of the important questions in topological dynamics is to evaluate how large is $M(G)$. 
While it is always nonmetrizable for non-compact locally compact groups, $M(G)$ can be metrizable for many dynamically massive groups. 

When $M(G)$ is trivial then $G$ is said to be \textit{extremely amenable}. 
This is equivalent to stating that every $G$-flow contains a $G$-fixed point. 
In some cases, $M(G)$ is dynamically small, constituting a tame $G$-flow or even a circularly ordered $G$-flow (e.g., for $G:=\Aut (\Q_0)$), \cite{GM-UltraHom}. For properties and the significance of tame and circularly ordered systems, see \cite{GM-TC,GM-UltraHom}.  

\sk 
\subsection*{Circular orders and their topology}
\label{s:circ} 
 

\begin{definition} \label{newC} \cite{Cech} 
	Let $X$ be a set. A ternary relation $R \subset X^3$ on $X$ is said to be a {\it circular} (or, sometimes, \emph{cyclic}) order  
	if the following four conditions are satisfied. It is convenient sometimes to write shortly $[a,b,c]$ instead of $(a,b,c) \in R$. 
	\ben
	\item Cyclicity: 
	$[a,b,c] \Rightarrow [b,c,a]$;  
	
	\item Asymmetry: 
	$[a,b,c] \Rightarrow (b, a, c) \notin R$; 
	
	\item Transitivity:    
	$
	\begin{cases}
		[a,b,c] \\
		[a,c,d]
	\end{cases}
	$ 
	$\Rightarrow [a,b,d]$;
	
	\item Totality: 
	if $a, b, c \in X$ are distinct, then \ $[a, b, c]$ 
	or $[a, c, b]$. 
	\een
\end{definition}

Observe that under this definition $[a,b,c]$ implies that $a,b,c$ are distinct. 

	On the set $C_n:=\{0, 1, \cdots, n-1\}$ consider the standard counter-clockwise cyclic order \textit{modulo $n$}.   
%
%
%
Let $X$ be a cyclically ordered set. 
We say that a vector $(a_1, \cdots, a_n) \in X^n$ 
(with $n \geq 3$) is a  \textit{cycle} in $X$ if 
$[a_i,a_j,a_k]$ for every c-ordered triple $[i,j,k]$ modulo $n$ in $C_n$.  Notation: $\nu=[a_1, \cdots, a_n]$.  
%

For 
$a,b \in X,$ define the (oriented) \emph{intervals}: 
$$
(a,b)_R:=\{x \in X: [a,x,b]\}, \ [a,b]_R:=(a,b) \cup \{a,b\}, \ [a,b)_R:=(a,b) \cup \{a\}, \ (a,b]_R:=(a,b)_R\cup\{b\}. 
$$
Clearly, 
$[a,a]=\{a\}$.  
Mostly we drop the subscript, or write $(a,b)_o$, $[a,b)_o$ and $(a,b]_o$ when $R$ is clear from the context.   

\begin{remark} \label{r:cech} \ \cite[p. 35]{Cech}
	\ben 
	\item 
	Every linear order $\leq$ on $X$ defines a \emph{standard circular order} $R_{\leq}$ on $X$ as follows: 
	$[x,y,z]$ iff one of the following conditions is satisfied:
	$$x < y < z, \ y < z < x, \  z < x < y.$$	
	
	\item (cuts) Let \((X,R)\) be a c-ordered set and let \(z\in X\). Define a linear
	order \(<_{z}\) on \(X\) by declaring \(z<_{z}x\) for every \(x\neq z\), and,
	for distinct \(a,b\in X\setminus\{z\}\),
	\[
	a<_{z}b \quad \Longleftrightarrow \quad [z,a,b].
	\]
	Then \(z\) is the least element of \((X,<_{z})\), and the circular order
	induced by \(<_{z}\) is the original circular order \(R\). Moreover, whenever
	\(z\notin (c,d)_R\), the linear interval \((c,d)_{<_{z}}\) coincides with the
	circular interval \((c,d)_R\).  
	\een
\end{remark}


\begin{proposition} \label{Hausdorff}  \cite{GM-UltraHom}   
Let $X$ be a set containing at least three elements. 
	For every c-order $R$ on $X$ the family of subsets
$${\mathcal B}:=\{(a,b)_R : \  a,b \in X, a \neq b\}$$
forms a base for a Hausdorff topology $\tau_R$ 
which we call the \emph{interval topology} of $R$.  

A topological space $(X,\tau)$ is said to be \emph{circularly ordered} if there exists a circular order $R$ such that $\tau_R=\t$.   
\end{proposition} 

Note that every $[a,b]_R$ is closed in the interval topology. Indeed, $X \setminus [a,b]_R =(b,a)_R$ for every distinct $a,b$. 

For more information about the topology of circular orders and dynamical applications for group actions, we refer to  \cite{GM-CircOrd18, GM-TC,GM-UltraHom,Me-OrdSem,Me-MaxEqComp,Me-b,Me-CircTop}. 

\sk 
\section{Inj-key subgroups which are not co-minimal} 
\label{s:example} 

In this section, we examine  a (non-minimal) Polish group $G:=\Aut (\Q_0)$ with metrizable $M(G)$ which was analyzed in 
\cite{GM-UltraHom} motivated by other reasons (like: generalized amenability of topological groups and circular orderability of $M(G)$).  
  
Let $\Q_0:=\Q/\Z$ be the rational circle with the usual counter-clockwise cyclic order on $\Q_0$. 
Consider the group $G:=\Aut (\Q_0)$
of all circular order preserving permutations of $\Q_0.$  
Note that the action $G \times \Q_0 \to \Q_0$ is \textit{circularly ultrahomogeneous}  (\textit{c-ultrahomogeneous}, in short), meaning that for every circular order preserving bijection between finite subsets $f \colon A \to B$ there exists $g \in G$ which extends $f$. 
Indeed, after cutting the circle at a point outside the finite sets under
consideration, this reduces to the usual ultrahomogeneity of the ordered set
\((\mathbb Q,<)\).

Now, take on $\Q_0$ the \textit{discrete topology} and consider the corresponding pointwise topology $\t_0$ 
on the group $G:=\Aut (\Q_0)$ inherited from the topological product $(\Q_0, discrete)^{\Q_0}$. 
Then
\(G := (\Aut(\mathbb Q_0),\tau_0)\) is a Polish \textit{non-archimedean} topological group,
that is, it has a local base at the identity consisting of open subgroups.

	Let $a, a_1, \cdots, a_n \in \Q_0$. 
We use the following notation for the stabilizer subgroups: 
$$G_a:=\{g \in G: \  g(a)=a\}, \ \  
G_{a_1, \cdots, a_n}:=
\cap_{i=1}^n G_{a_i}.$$  
Choose $a_0 \in \Q_0$. Then the stabilizer 
$$H:=G_{a_0}=\{g \in \Aut (\Q_0): \ ga_0=a_0 \}$$ is a clopen  subgroup of $G$. It is topologically isomorphic to $\Aut(\Q,\leq)$.  
This Polish group $\Aut(\Q,\leq)$ is extremely amenable
by an important well-known result of Pestov \cite{Pest98}.    
The coset $G$-space $(G/H, \t_0/H)$ is topologically the discrete copy of $\Q_0$ with the tautological action of $G=\Aut (\Q_0)$. 

\begin{theorem} \label{t:counterexample} The subgroup 
	$H:=G_{a_0} \simeq \Aut(\Q,\leq)$ is an inj-key subgroup but not a co-minimal subgroup in $G:=(\Aut (\Q_0), \t_0)$. 
\end{theorem}

First, observe that $H$ is not co-minimal, as demonstrated by Lemma \ref{l:HisnotCo-min} below.  
In order to prove that $H=G_{a_0}$ is inj-key in $G$, we need first to obtain several intermediate helpful results and recall related  notions. A part of them are taken from \cite{GM-UltraHom}. 	


	\begin{lemma} \label{l:stabilizers} \ 
		\begin{enumerate} 
			\item $gG_ag^{-1}=G_{g(a)}$ for every $g \in G$.  
			 \item $Core_G(H):=\cap\{gHg^{-1}: \ g \in G\}=\{e\}$ and hence the Cayley homomorphism $G \to S_{G/H}$ into the symmetric group $S_{G/H}$ is injective. 
			\item The family of open subgroups $G_{a_1, \cdots, a_n}$ is a local base of the (non-archimedean) group topology $\t_0$ on $G$.  
			\item 	Let $\s \in \mathcal{T}_{\downarrow}(G, \t_0)$. Assume that there exists a point $a \in \Q_0$ such that $G_a \in \s$. Then $\s=\t_0$. 
  \item $G_a$ is a maximal subgroup of $G$ for every $a \in \Q_0$.  
  \item For every $z \in \Q_0$ the induced action of $G_{z}$ on the linearly ordered set $(\Q_0 \setminus \{z\},<_z)$ is linear order preserving and ultrahomogeneous.
  \item 
   Let \(c,d\in \mathbb Q_0\) be distinct elements and choose
   \(z\in \mathbb Q_0\setminus (c,d)_R\). Define
   \[
   M_{[c,d]}:=\{g\in G: gx=x \ \forall x\notin (c,d)_R\}.
   \]
   Then \(M_{[c,d]}\) is a subgroup of \(G\). Its restriction to
   \((c,d)_R=(c,d)_{<_{z}}\) is naturally isomorphic to
   \(\Aut((c,d)_{<_{z}},<_{z})\), and therefore \(M_{[c,d]}\) acts linearly
   ultrahomogeneously on this interval. 
		\end{enumerate} 
	\end{lemma}
	\begin{proof} (1), (2), (3) are straightforward. 
	
	(4) By (1) and the transitivity of the \(G\)-action on \(\mathbb Q_0\), every
	point stabilizer \(G_b\), \(b\in\mathbb Q_0\), is \(\sigma\)-open. Hence every
	finite intersection \(G_{b_1,\ldots,b_n}\) is \(\sigma\)-open. Since these
	subgroups form a local base of \(\tau_0\) at the identity, we get
	\(\tau_0\subseteq\sigma\). As \(\sigma\subseteq\tau_0\), it follows that
	\(\sigma=\tau_0\).

(5) Let \(G_a\leq P\leq G\) and suppose that \(G_a\neq P\). Choose
\(g\in P\setminus G_a\), so \(g(a)\neq a\). We show that \(P=G\). Let
\(f\in G\). If \(f(a)=a\), then \(f\in G_a\subseteq P\). Otherwise, by
circular ultrahomogeneity there exists \(\varphi\in G_a\) such that
\(\varphi(f(a))=g(a)\). Hence \(g^{-1}\varphi f\in G_a\), and therefore
\[
f\in \varphi^{-1}gG_a\subseteq G_aPG_a=P.
\]
Thus \(P=G\), proving that \(G_a\) is maximal.

%
%
%
%
		
		(6) It is a partial case of a claim from \cite[Theorem 4.3]{GM-UltraHom}. 
		
		(7) Observe that the linearly ordered set $(c,d)_{<_z}$ is linearly isomorphic to $(\Q, \leq)$ and the group $M_{[c,d]}$ can be treated as $\Aut(\Q, \leq)$.  
		
	\end{proof}

	It is well known (see, for example, \cite[Example 4.2]{BT}) that there exists an injective continuous dense homomorphism $\Aut(\Q, \leq) \to \Homeo_+[0,1]$ of Polish groups. 
	A similar embedding remains true for the rational circle  $\Q_0$ (see 
	\cite[Section 5.1]{CG}).

	
\begin{lemma} \label{l:emb} 
	Let $\Homeo_+(\T)$ be the group of all circular order preserving autohomeomorphisms of the circle $\T$. 
	Consider on $\Homeo_+(\T)$ 
	its usual compact-open topology $\t_{co}$.  
Then the following hold. 
\begin{enumerate}
	\item There exists a continuous injective $\t_{co}$-dense  homomorphism  
	$$i_{\T} \colon G=(\Aut (\Q_0), \t_0) \to (\Homeo_+(\T), \t_{co}).$$
	\item 
	$i_{\T}(G_q)=St(q) \cap i_{\T}(G)$, where $St(q)$ is the stabilizer subgroup 
	$$St(q): =\{g \in \Homeo_+(\T) : \ g(q)=q\}$$ 
	of $\Homeo_+(\T)$ for every point $q \in \Q_0$. In particular, $i_{\T}(H)=St(a_0) \cap i_{\T}(G)$. 
\end{enumerate}
\end{lemma} 
	\begin{proof}
Notation $\Homeo_+(\T, \Q_0):=\{f \in \Homeo_+(\T):  f(\Q_0)=\Q_0\}.$  
For every \(g\in \Aut(\mathbb Q_0)\), since \(g\) preserves the circular order
on the dense subset \(\mathbb Q_0\subset \mathbb T\), there exists a unique
orientation-preserving homeomorphism \(\bar g\in \Homeo^+(\mathbb T)\)
extending \(g\). Moreover, \(\bar g(\mathbb Q_0)=\mathbb Q_0\). Define
\[
i_{\mathbb T} \colon \Aut(\mathbb Q_0)\to \Homeo^+(\mathbb T), \qquad
i_{\mathbb T}(g)=\bar g .
\]
Then \(i_{\mathbb T}\) is an injective homomorphism. Its continuity follows
from the standard fact that, for monotone circle homeomorphisms, pointwise
convergence on the dense set \(\mathbb Q_0\) implies compact-open convergence
on \(\mathbb T\). Its image is
\[
i_{\mathbb T}(G)=\Homeo^+(\mathbb T,\mathbb Q_0)
:=\{f\in \Homeo^+(\mathbb T): f(\mathbb Q_0)=\mathbb Q_0\},
\]
which is dense in \(\Homeo^+(\mathbb T)\). 

Finally, for every \(q\in\mathbb Q_0\),
$
i_{\mathbb T}(G_q)=St(q)\cap i_{\mathbb T}(G),
$
where \(St(q):=\{f\in \Homeo^+(\mathbb T): f(q)=q\}\). 

\end{proof}

	Denote by $\t_{\T}$ the (preimage) topology on $G$ induced by $i_{\T}$. Clearly, $\t_{\T} \in  \mathcal{T}_{\downarrow}(G,\t_0)$. 
	Sometimes we identify $G$ algebraically with the dense subgroup $i_{\T}(G)$ of $\Homeo_+(\T)$. 
	
	\begin{lemma} \label{l:HisnotCo-min} 
		$H:=G_{a_0}$ 
		is not co-minimal in $G$.  
	\end{lemma}
	\begin{proof}   
		Clearly, the coset space $\Homeo_+(\T)/St(a_0)$ is the circle $\T$, where $St(a_0)=\{g \in \Homeo_+(\T):  g(a_0)=a_0\}$. 
The subgroup \(i_{\mathbb T}(H)=St(a_0)\cap i_{\mathbb T}(G)\) is not open in
\(i_{\mathbb T}(G)\). Indeed, every compact-open neighborhood of the identity
in \(i_{\mathbb T}(G)\) contains elements moving \(a_0\) slightly, for instance
to nearby rational points of \(\mathbb Q_0\).
	
	Thus, $i_{\T}(G) / i_{\T}(H)$ is not discrete. This means that for the topology $\t_{\T}$ on $G$ the corresponding coset topology $\t_{\T}/H$ is not discrete (in contrast, to $\t_0/H$). Therefore, 
		$H$ is not co-minimal in $G$.    
		
	\end{proof}
	

	Recall the following useful result due to Chang--Gartside \cite{CG}. 
	
	\begin{fact} \label{f:Thm8} \cite[Theorem 8]{CG} 
Let $X$ be a compact metrizable space. Denote by $\mathcal{O}_X$ the subset of all points $x \in X$ such that there exists 
a neighborhood $O_x$ of $x$ in $X$ which is homeomorphic to the real interval $(0,1)$.  
Let \(C_X:=X \setminus \mathcal O_X\). 
Assume that $\mathcal{O}_X$ is dense in $X$. 
Let $G$ be a subgroup of $\Homeo (X)$ such that 
for every open subset $U$ of $X$ homeomorphic to $(0, 1)$ there is a non-trivial $g \in G$ with 
$$Move(g):=\{x \in X: \ gx \neq x\} \subseteq U.$$ 
Then there exists a canonically defined topology $\t_{co}|C_X$  on $\Homeo (X)$ such that this topology 
restricted to $G$ is the \emph{minimum} Hausdorff group topology on 
the abstract group $G$ (i.e., it is the least Hausdorff group topology on the discrete group $G$). 		
\end{fact}

	
For simplicity we do not include definitions from \cite{CG} describing the canonically defined topology $\tau_{co}|_{C_X}$. 
It is important to note that for empty $C_X$ (that is, for $\mathcal{O}_X=X$) this topology is the usual compact-open topology $\t_{co}$. This happens for example in the case of the circle $X:=\T$. See also a discussion about the ``rational circle" in  \cite[Section 5.1]{CG}.

\begin{lemma} \label{l:minimum}    
The compact-open topology $\t_{\T}$ is the \textbf{minimum} group topology on $G:=\Aut (\Q_0)$. 
		Hence, $\t_{\T}$ is the least element in $\mathcal{T}_{\downarrow}(G,\t_0)$. Also, $H$ is $\s$-closed in $G$ for every $\s \in \mathcal{T}_{\downarrow}(G,\t_0)$ (therefore, $H$ is a \emph{strongly closed} subgroup of $G$ in terms of \cite{DM10}). 	
\end{lemma}	
\begin{proof} 
	Identifying \(G\) with \(i_{\mathbb T}(G)\), we have a dense subgroup of
	\(\Homeo_+(\mathbb T)\). It satisfies the local movement assumption of
	Fact~\ref{f:Thm8}: indeed, if \(U\subset \mathbb T\) is an open arc, choose
	\(c,d\in\mathbb Q_0\) with \([c,d]_{\mathbb T}\subset U\), and take a nontrivial
	element of \(M_{[c,d]}\). Its extension to \(\mathbb T\) has
	\(Move(g)\subset U\). 
	Therefore,  
	the first part of this lemma can be derived from 
	Fact \ref{f:Thm8}. 
	
	For the second statement, observe that
	\(i_{\mathbb T}(H)=St(a_0)\cap i_{\mathbb T}(G)\) by Lemma~\ref{l:emb}.2,
	and \(St(a_0)\) is closed in the compact-open topology. Hence \(H\) is
	\(\tau_{\mathbb T}\)-closed. Since \(\tau_{\mathbb T}\subseteq\sigma\) for
	every \(\sigma\in\mathcal T_{\downarrow}(G,\tau_0)\), it follows that \(H\)
	is \(\sigma\)-closed for every such \(\sigma\). 
	 
\end{proof} 


Let $\nu:=[a_1, \cdots, a_n]$ (with $n \geq 3$) be a \textit{cycle} in the original  circularly  ordered  
 set $\Q_0$. Define the following disjoint finite covering of $G/H=\Q_0$: 
\begin{equation} \label{eq:cov}  
	cov(\nu):=\{
	\{a_1\}, (a_1,a_2), \{a_2\}, (a_2,a_3), \cdots, \{a_n\},  (a_n,a_1)\}. 
\end{equation}

\begin{fact}  \label{f:GM} \cite{GM-UltraHom}  \   
	\begin{enumerate}
		\item The natural \textit{right uniformity} $\mu_r$ on $G/H=\Q_0$ is precompact, containing a uniform base $B_r$,  where its typical element is the disjoint covering $cov(\nu)$. The completion of $\mu_r$ is the greatest $G$-compactification $\beta_G (G/H)$ of the discrete $G$-space  $G/H=\Q_0$. 
	
		\item $\beta_G (G/H)= \beta_G (\Q_0) =\widehat{G/H}=trip(\T,\Q_0)$. 
		
		\noindent Geometrically, $\beta_G (G/H)=trip(\T,\Q_0)$ is a circularly ordered metrizable compact space  which we get from the ordinary circle $\T$ after replacing any rational point $q\in \Q_0$ by the ordered triple of points $q^{-}, q, q^{+}$. Namely, we have $[q^{-}, q, q^{+}]$ in $\beta_G (G/H)$. 
		The action is given by
		$
		g(q^-)=g(q)^-,\ g(q)=g(q),\ g(q^+)=g(q)^+
		$
		for every \(q\in\mathbb Q_0\) and \(g\in G\).
			\begin{figure}[h]  
			\begin{center} 
				\scalebox{0.4}{\includegraphics{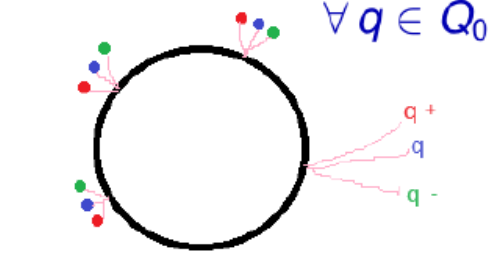}}
				\caption{Geometric description of $\beta_G (\Q_0)$}
			\end{center} 
			\label{fig:two}
		\end{figure}  
		\item $\beta_G (\Q_0)=M(G) \cup \Q_0$ and $M(G)=split(\T,\Q_0)$ is the universal minimal $G$-flow of $G$. 

\noindent 
Geometrically, the compact subflow \(M(G)\) is obtained from \(\mathbb T\) by
splitting each rational point \(q\in\mathbb Q_0\) into the two points
\(q^-\) and \(q^+\); equivalently, it is the subspace of \(trip(\mathbb T,\mathbb Q_0)\)
obtained by deleting the middle points \(q\). 
	\end{enumerate}	
\end{fact}

\begin{lemma} \label{l:orbit} \ 
	\begin{enumerate}
		\item There are four orbits of the action $G \times \beta_G (\Q_0) \to \beta_G (\Q_0)$. 
		Namely: 
		  
		\noindent a) $\Q_0$; b) $\Q^{-}_0$; c) $\Q^{+}_0$; d) $J:=\beta_G (\Q_0) \setminus K$, where $K: = \Q_0 \cup \Q^{-}_0 \cup \Q^{+}_0$.   
		
		\item  The induced action of $G$ on each of the four orbits is c-ultrahomogeneous. 
	\end{enumerate} 
\end{lemma}
\begin{proof}  (1) 
	$\widehat{G/H}$ is the completion of the precompact uniform space $(G/H,\mu_r)$ (Fact \ref{f:GM}.1). 
	Every element $u \in \widehat{G/H}$ can be identified with a \textit{minimal Cauchy-filter} in $G/H=\Q_0$ (see \cite{Bourb}).  
	Sometimes it is more convenient to use a suitable filter base $\xi_u$ such that the corresponding (uniquely defined)  minimal Cauchy filter represents $u$. 
	\sk 
		a) 
	The most clear case is the element $u=q\in \Q_0$. Each of them is isolated and the singleton $\xi_q:=\{q\}$ is a filter base which represents the point $q$.  
	
	\sk 
	For any $u \notin \Q_0$,  
	there exists a suitable $\mu_r$-Cauchy filter base $\xi_u$  such that it is a decreasing sequence of intervals $\xi_u:=\{(a_n,b_n) : n \in \N\}$. That is, $a_n, b_n \in \Q_0$ and 
	$$(a_n,b_n)  \supseteq (a_{n+1}, b_{n+1})$$ for every $n \in \N$.
	Since it is a $\mu_r$-Cauchy filter base and $d|_{\Q_0}$ is $\mu_r$ uniform, we have $\lim_{n \in \N} d(a_n,b_n)=0$, where $d$ is a metric of the circle $\T$. 
	
	Conversely, any decreasing sequence of intervals with this smallness condition implies that this sequence $\xi_u$ is a  $\mu_r$-Cauchy filter base.   
	
	\sk 
	b) For every $u=q^{-}$ consider any decreasing sequence of intervals  $(a_n,b_n)=(q,b_n)$ in $\Q_0$ such that (all $a_n=q$ and)  $\lim_{n \in \N} d(q,b_n)=0$. 
	
	\sk 
	c) For every $u=q^{+}$ consider any decreasing sequence of intervals  $(a_n,q)$ in $\Q_0$ such that (all $b_n=q$ and)  
	$\lim_{n \in \N} d(a_n,q)=0$.  
	
	\sk 
	
	d) Now, let $u$ be ``irrational", that is, $u \in J$. 
	Then there exists a suitable decreasing sequence $\xi_u=\{(a_n,b_n)\}$ of intervals with 
	 $\lim d(a_n,b_n)=0$ such that, in addition, 
	 $a_n \neq a_{n+1}, \ b_n \neq b_{n+1}$ for every $n \in \N$.    
	
\sk 
	First three $G$-orbits $\Q_0$,  $\Q^{-}_0$, $\Q^{+}_0$ are understood. 
	
	We show 
	that $J=\beta_G (\Q_0) \setminus K$ is a $G$-orbit.  
	 Let $u, v \in J$ be two distinct irrational elements.   
	 Consider the corresponding decreasing sequences of intervals (which act as filter bases, as described in (d)): 
	 
	 $$\xi_u=\{(a_n,b_n): n \in \N\}, \ \ \ \xi_v= \{(c_n,d_n): n \in \N\}.$$
	 
	 Since \(u \neq v\), we can assume that
	 \([a_1,b_1] \cap [c_1,d_1]=\emptyset\). 
	  Our aim is to show that there exists $g \in G$ such that $gu=v$. We will explore the Raikov completeness of $G$ constructing the desired $g$ as the limit of a certain Cauchy sequence $g_n$ (with respect to the two-sided  uniformity ${\mathcal U}_{l \vee r}$ of $G$).  We will construct $g$ inductively using c-ultrahomogeneity of $\Q_0$.  
	 First take $g_1 \in G$ such that 
	 $$
	 g_1 [a_1,b_1] =[c_1,d_1].   
	 $$ 

Using Lemma~\ref{l:stabilizers}.7, choose
\(h_2\in M_{[c_1,d_1]}\) such that
\[
h_2g_1[a_2,b_2]=[c_2,d_2].
\]
Set \(g_2:=h_2g_1\). Since \(h_2\) fixes every point outside
\((c_1,d_1)\), and \(g_1\) maps the complement of \((a_1,b_1)\) onto the
complement of \((c_1,d_1)\), we have
\[
g_2(x)=g_1(x)\quad \forall x\notin (a_1,b_1).
\]
Also,
\[
g_2^{-1}(x)=g_1^{-1}(x)\quad \forall x\notin (c_1,d_1).
\]

Inductively, suppose that \(g_{n-1}[a_{n-1},b_{n-1}]=[c_{n-1},d_{n-1}]\).
Choose \(h_n\in M_{[c_{n-1},d_{n-1}]}\) such that
\[
h_ng_{n-1}[a_n,b_n]=[c_n,d_n],
\]
and set \(g_n:=h_ng_{n-1}\). Then
\[
g_n[a_n,b_n]=[c_n,d_n],
\]
and
\[
g_n(x)=g_{n-1}(x)\quad \forall x\notin (a_{n-1},b_{n-1}),
\]
while
\[
g_n^{-1}(x)=g_{n-1}^{-1}(x)\quad \forall x\notin (c_{n-1},d_{n-1}).
\]
It follows that, for every \(i,j>n\),
\begin{equation}\label{eq:cauchy-gn}
	g_i(x)=g_j(x)\quad \forall x\notin (a_n,b_n),
	\qquad
	g_i^{-1}(x)=g_j^{-1}(x)\quad \forall x\notin (c_n,d_n).
\end{equation}

We claim that \((g_n)\) is a Cauchy sequence in \(G\) with respect to the
two-sided uniformity \(\mathcal U_{l\vee r}\). Let
\(G_{t_1,\ldots,t_k}\) be a basic neighbourhood of \(e\) in \(G\). Since
\(\bigcap_n [a_n,b_n]\) and \(\bigcap_n [c_n,d_n]\) are irrational points,
we may choose \(n_0\) such that
\[
\{t_1,\ldots,t_k\}\cap [a_{n_0},b_{n_0}]=\emptyset,
\qquad
\{t_1,\ldots,t_k\}\cap [c_{n_0},d_{n_0}]=\emptyset .
\]
Then, for every \(n,m>n_0\), Equation~\ref{eq:cauchy-gn} gives
\[
g_n(t_i)=g_m(t_i)
\quad\text{and}\quad
g_n^{-1}(t_i)=g_m^{-1}(t_i)
\qquad (i=1,\ldots,k).
\]
Equivalently,
\[
g_n^{-1}g_m\in G_{t_1,\ldots,t_k}
\quad\text{and}\quad
g_ng_m^{-1}\in G_{t_1,\ldots,t_k}.
\]
Thus \((g_n)\) is two-sided Cauchy. Since \(G\) is Polish, hence Raikov
complete, there exists \(g\in G\) such that \(g_n\to g\) in the two-sided
uniformity. By construction, \(g[a_n,b_n]=[c_n,d_n]\) for every \(n\), in
the sense that the Cauchy filter base \(\{(a_n,b_n)\}\) is sent to the
Cauchy filter base \(\{(c_n,d_n)\}\). Therefore \(gu=v\).

\sk 
(2) The induced actions of $G$ on the orbits (as circularly ordered sets) $\Q_0^{-}$ or $\Q_0^{+}$ are circularly identical with the original action of $G$ on $\Q_0$. 
As to the action of $G$ on $J$, let  
\[
f \colon U=\{u_1,\ldots,u_m\}\to V=\{v_1,\ldots,v_m\}
\]
 be a circular isomorphism between finite subsets of $J$. 
For every $u_i$ and $v_i$ choose the corresponding filter bases 
$\xi_{u_i}, \xi_{v_i}$ in such a way that the initial intervals 
$[a_{1i},b_{1i}]$ of all $u_i$-s are pairwise disjoint and similarly the 
initial intervals 
$[c_{1i},d_{1i}]$ of all $v_i$-s are pairwise disjoint. 

Applying the preceding inductive construction (similar to the part (1)) simultaneously on these pairwise
disjoint initial intervals, and using circular ultrahomogeneity at each finite
stage, we obtain \(g\in G\) such that \(g(u_i)=v_i\) for all \(i=1,\ldots,m\).  

\end{proof}

\begin{lemma} \label{l:maximality}  
	For every $z \in \beta_G (\Q_0)$ the stabilizer subgroup $G_z$ (defined for the action of $G$ on $\beta_G (\Q_0)$) is maximal in $G$. 
\end{lemma}
\begin{proof}
	If $z \in \Q_0$ then apply Lemma \ref{l:stabilizers}.5. 
	
	If \(z=q^-\in \Q^-_0\) or \(z=q^+\in \Q^+_0\), where \(q\in \Q_0\), then
	$
	G_{q^-}=G_q=G_{q^+},
	$
	by the definition of the action \(g(q^-)=g(q)^-\) and \(g(q^+)=g(q)^+\) on \(\beta_G(\Q_0)\)  from
	Fact~\ref{f:GM}. 
	
	Assume that $z \in J$ is irrational. 
	In the proof of Lemma \ref{l:orbit} we showed that the subset of all irrational points $J$ is a $G$-orbit. An easy modification shows that for every distinct points $z, u, v  \in J$ there exists $g \in G$ such that $gz=z, gu=v$. 
	That is, $G_z$ is transitive on $J \setminus \{z\}$. 
	One may choose corresponding filter bases:  
	$$\xi_z=\{(s_n,t_n)\}, \ \ \ \xi_u=\{(a_n,b_n): n \in \N\}, 
	\ \ \ \xi_v= \{(c_n,d_n): n \in \N\}$$
	such that $[s_1,t_1] \cap  [a_1,b_1]=\emptyset$, $[s_1,t_1] \cap [c_1,d_1]=\emptyset$, $[a_1,b_1] \cap [c_1,d_1]=\emptyset$. Under this requirement, one may choose $g_1 \in G$ such that $g_1$ is the identity on $[s_1,t_1]$ and $g_1[a_1,b_1]=[c_1,d_1]$. The rest is similar to the proof of Lemma \ref{l:orbit}.     
	
	Now, using this transitivity property of $G_z$, we prove the maximality of $G_z$ in $G$ using the arguments that  we explored in the proof of Lemma \ref{l:stabilizers}.5 (this time we consider the action of $G$ on the circularly ordered set $J$).   
	
	Indeed, let \(G_z\leq P\leq G\) and suppose that \(P\neq G_z\). Choose
	\(p\in P\setminus G_z\). Then \(p(z)\neq z\). We show that \(P=G\).
	Let \(f\in G\). If \(f(z)=z\), then \(f\in G_z\subseteq P\). Otherwise,
	by the transitivity of \(G_z\) on \(J\setminus\{z\}\), there exists
	\(\varphi\in G_z\) such that
	\[
	\varphi(f(z))=p(z).
	\]
	Hence \(p^{-1}\varphi f\in G_z\), and therefore
	\[
	f\in \varphi^{-1}pG_z\subseteq G_zPG_z=P.
	\]
	Thus \(P=G\), proving that \(G_z\) is maximal in \(G\). 	
	
\end{proof}

\begin{remark} \label{r:altern} 
	Let \(G/H\) be an abstract transitive \(G\)-space and let
	\(f:G/H\to Y\) be an onto \(G\)-map. Put \(y_0:=f(H)\) and
	\(H_1:=G_{y_0}=\operatorname{St}_G(y_0)\). Then \(H\leq H_1\leq G\), and \(Y\) is naturally
	\(G\)-isomorphic to \(G/H_1\). Consequently, if \(H\) is a maximal subgroup
	of \(G\), then either \(H_1=H\), in which case \(f\) is bijective, or
	\(H_1=G\), in which case \(f\) is constant.
\end{remark}

Assertion (a) of Lemma \ref{l:list} was proved (using different arguments) in \cite{GM-UltraHom}. 

\begin{lemma} \label{l:list} \ 
	\begin{itemize}
		\item [(a)]  \cite{GM-UltraHom}
		The \(G\)-flow
		$
		M(G)=\mathbb Q_0^-\cup J\cup \mathbb Q_0^+
		$
		has, up to equivalence, only one nontrivial proper \(G\)-factor, namely the
		factor onto \(\mathbb T\) obtained by identifying \(q^-\) with \(q^+\) for
		every \(q\in\mathbb Q_0\). Thus the only \(G\)-factors of \(M(G)\) are the singleton, \(\mathbb T\),
		and \(M(G)\) itself.
		\item [(b)] 
	Consider the following $G$-factors of $\beta_G (\Q_0)=\Q^- \cup J \cup \Q^+ \cup \Q_0$. 	
	\begin{enumerate}
		\item $f_1 \colon \beta_G (\Q_0) \to \T$ with $f_1(q^{-})=f_1(q)=f_1(q^{+})=q$ and $f_1(x)=x$ for every 
		$q \in \Q_0$ and every $x \in J$. 
		\item $f_2 \colon \beta_G (\Q_0) \to \Q_0 \cup \{*\}$,  where $\Q_0 \cup \{*\}$ is the one-point  compactification of the discrete set $\Q_0$. In this case $f_2(M(G))=*$ and $f_2(q)=q$ for every $q \in \Q_0$.  
		\item $f_3 \colon \beta_G (\Q_0) \to M(G)$ with $f_3(q)=q^{-}$, $f_3(x)=x$ for every $q \in \Q_0, x \in M(G)$.   
		\item $f_4 \colon \beta_G (\Q_0) \to M(G)$ with $f_4(q)=q^{+}$, $f_4(x)=x$ for every $q \in \Q_0, \ x \in M(G)$.   
				\item $f_5 \colon \beta_G (\Q_0) \to \T \cup \Q_0$ with $f_5(q^{+})=f_5(q^{-})=i(q)$ for every $q \in \Q_0$ and $f_5(x)=x$ for every $x \notin  (\Q^{-}_0 \cup \Q^{+}_0)$.  Here $i(\Q_0)$ is the usual 
				rational part of the circle $\T$.  
		\item $f_6:=id \colon \beta_G (\Q_0) \to \beta_G (\Q_0)$. 
		
		\sk 
		Each of these maps is continuous, \(G\)-equivariant, and injective on the
		middle orbit \(\mathbb Q_0\). 
		Let $\nu \colon \Q_0 \to Y$ be a (not necessarily, proper) \textbf{injective} $G$-compactification. Then $\nu$ is equivalent to one of the compactifications $f_i|_{\Q_0}$, where $i \in \{1,\cdots,6\}$.  
	\end{enumerate} 
		\end{itemize}
	
\end{lemma}
\begin{proof}   
	We start with a simple observation which will be used several times.
	
	\smallskip
	\noindent\textit{Claim.}
	Let \(q\in \mathbb Q_0\) and \(u\in J\). Then
	\[
	G_{q^-}=G_q=G_{q^+},
	\]
	and this subgroup fixes exactly one point of \(\mathbb Q_0\), namely \(q\).
	On the other hand, \(G_u\) fixes no point of \(\mathbb Q_0\). Consequently,
	\(G_u\) is not conjugate to any of the subgroups
	\(G_q=G_{q^-}=G_{q^+}\).
	
	\sk 
	Indeed, the equality \(G_{q^-}=G_q=G_{q^+}\) follows from the definition of
	the action on \(trip(\mathbb T,\mathbb Q_0)\). Clearly \(G_q\) fixes \(q\);
	and, by circular ultrahomogeneity, no other point of \(\mathbb Q_0\) is fixed
	by all elements of \(G_q\). If \(u\in J\), then cutting the circle at \(u\)
	identifies \(\mathbb Q_0\) with a countable dense linearly ordered set without
	endpoints, and the induced action of \(G_u\) on this set is the usual
	ultrahomogeneous action of \(\Aut(\mathbb Q,<)\). Hence \(G_u\) has no fixed
	point in \(\mathbb Q_0\). Since conjugation preserves the fixed-point set in
	\(\mathbb Q_0\), the asserted non-conjugacy follows.
	\smallskip

	(a) Let \(f \colon M(G)\to Y\) be a nontrivial \(G\)-factor. Since \(M(G)\) is
	minimal, \(Y\) is minimal. The three \(G\)-orbits
	\(\mathbb Q_0^-\), \(J\), and \(\mathbb Q_0^+\) are dense in \(M(G)\).
	Hence the restriction of \(f\) to any of these orbits is either injective
	or constant, by Remark~\ref{r:altern} and Lemma~\ref{l:maximality}. If the
	restriction to one of these dense orbits is constant, then \(Y\) is a
	singleton. Thus, for a nontrivial factor, \(f\) is injective on each of the
	three orbits. 
	It remains to understand possible identifications between different orbits.
	
	There is no nonconstant \(G\)-map from \(J\) onto either
	\(\mathbb Q_0^-\) or \(\mathbb Q_0^+\). Indeed, such a map would force,
	by maximality of point stabilizers, \(G_u\) to be conjugate to
	\(G_{q^-}=G_q=G_{q^+}\), contradicting the claim.
	
 On the other hand, the two orbits \(\mathbb Q_0^-\) and \(\mathbb Q_0^+\) are equivariantly
	identified by the map \(q^-\mapsto q^+\), since
	\(G_{q^-}=G_q=G_{q^+}\). Therefore the only possible proper nontrivial
	identification is
	$
	q^-\sim q^+ \ (q\in\mathbb Q_0),
	$
	and the resulting factor is the usual circle \(\mathbb T\).

	(b) By Lemma~\ref{l:orbit}, the action of \(G\) on \(\beta_G(\mathbb Q_0)\)
	has exactly four orbits:
	\[
	\mathbb Q_0,\qquad \mathbb Q_0^-,\qquad \mathbb Q_0^+,\qquad J .
	\]
	Let \(f:\beta_G(\mathbb Q_0)\to Y\) be a \(G\)-factor which is injective on
	\(\mathbb Q_0\). By Remark~\ref{r:altern} and Lemma~\ref{l:maximality}, the
	restriction of \(f\) to each orbit is either injective or constant. Since
	\(f\) is injective on \(\mathbb Q_0\), the orbit \(\mathbb Q_0\) is not
	collapsed.
	
	If one of the three orbits contained in \(M(G)\) is collapsed to a point,
	then, since \(M(G)\) is minimal, the whole subflow \(M(G)\) is collapsed to
	one point. This gives case \(f_2\).
	
	Assume now that no orbit in \(M(G)\) is collapsed. Then \(f\) is injective
	on each of the four orbits. Hence a proper factor can only arise from
	identifying points lying in different orbits. Such identifications must be
	\(G\)-invariant. The only possible equivariant identifications between
	different orbits are the following rational twin identifications:
	\[
	q^-\sim q^+,\qquad q\sim q^-,\qquad q\sim q^+
	\quad (q\in\mathbb Q_0).
	\]
	By the claim, there is no nonconstant \(G\)-map from \(J\) onto any of
	\(\mathbb Q_0,\mathbb Q_0^-,\mathbb Q_0^+\), since the corresponding point
	stabilizers are not conjugate. 
	Thus the possible proper quotients are only: 
	\begin{itemize}
		\item 
		$q^-\sim q^+ \quad\text{for all }q,
		$ which gives \(f_5\);
		\item $q\sim q^-\quad\text{for all }q,$ 
		which gives \(f_3\);
		\item $q\sim q^+\quad\text{for all }q,$ 
		which gives \(f_4\); 
		\item $
		q^-\sim q\sim q^+\quad\text{for all }q,
		$ 
		which gives \(f_1\). 
	\end{itemize} 
	Together with the case \(f_2\), where \(M(G)\) is
	collapsed, and the identity case \(f_6\), this gives exactly the six
	compactifications listed above. 
	
\end{proof}

Note that topologically $\T \cup \Q_0$, in item (5), is a natural subspace of the classical (Alexandrov)  \textit{double circle} $\T \cup \T_{discr}$, where $\Q_0$ is a subset of the discrete copy $\T_{discr}$ of $\T$. 

Many remarkable results about automorphism groups of ordered spaces were obtained recently by B.V. Sorin, G.B. Sorin and K.L. Kozlov.  
Maximal $G$-compactifications of 
ultratransitive $G$-actions on discrete circularly (and linearly) ordered sets and their $G$-factors,  admit a general description \cite{Sorin25-Sbornik,KS}.

\sk

Lemma \ref{l:list} gives a list of all possible $G$-factors of $\beta_G (\Q_0)$ (which are injective on $\Q_0=G/H$). 
Hence, one may apply Lemma \ref{l:CosetTopologies} getting important information about all admissible topologies on 
$(G/H,\s/H)$, where $\s \in \mathcal{T}_{\downarrow}(G,\t_0)$.  

The following result asserts that,  
%
%
%
%
in our setting, we have only two possibilities for ``admissible" topologies on $\Q_0=G/H$. One is the original discrete topology $\t_0/H$. The second is the topology coming from the embedding $i_{\T}$ (and compact-open topology) described in Lemma \ref{l:emb}.  

\begin{lemma} \label{l:equalities} 
	$
	\{\s/H: \  \s \in \mathcal{T}_{\downarrow}(G,\t_0)\} = 	\{\g/H: \  \g \in \Sigma_H(G)\}=\{\t_0/H, \ \t_{\T}/H \}. 
	$
	More precisely,
	\begin{enumerate}
		\item  For every $\g \in \Sigma_H(G),$ the equality  $\g/H=\t_{\T}/H$ holds only for the first factor (1) in the list of Lemma \ref{l:list}. 
		\item 	In cases (2), (3), (4), (5), (6) of Lemma \ref{l:list}, the corresponding compact-open topology $\g$ on $G$ is the original non-archimedean  topology $\t_0$.   
	\end{enumerate}  
\end{lemma}
\begin{proof} 
The first equality directly comes from Lemma \ref{l:CosetTopologies}. Now, it is enough to show that 
$\g/H \in \{\t_0/H, \ \t_{\T}/H \}$ for every $\g \in \Sigma_H(G)$.

Let $f_i \colon \beta_G (\Q_0) \to  Y$ be the quotients from Lemma \ref{l:list}.b. 
We have the corresponding compact-open topologies $\g_i \in \Sigma_H(G)$ on $G$.   
Let \(f_i:\beta_G(\Q_0)\to Y_i\) be one of the quotients from
Lemma~\ref{l:list}. 
We denote by \(\gamma_i\) the corresponding
compact-open topology on \(G\).

For \(i=1\), the compactification is the canonical circle compactification.
Hence the induced topology on \(G/H=\Q_0\) is exactly \(\tau_{\mathbb T}/H\).

We now show that \(\gamma_i=\tau_0\) for \(i=2,3,4,5,6\). For \(i=2,5,6\),
the compact \(G\)-space \(Y_i\) contains the orbit \(\Q_0\) as a discrete
\(G\)-subspace. Hence, for every \(a \in \Q_0\), the stabilizer \(G_a\) is 
\(\gamma_i\)-open. Since \(\gamma_i\subseteq\tau_0\), Lemma~\ref{l:stabilizers}.4
implies \(\gamma_i=\tau_0\).

It remains to consider the twin cases \(i=3,4\), where \(Y_i=M(G)\). Let 
\(a,b\in \Q_0\) be distinct. In \(M(G)=\Q_0^-\cup J\cup \Q_0^+\), the interval
$
A:=[a^+,b^-]=(a^-,b^+)
$
is clopen. Therefore its setwise stabilizer
$
G_A:=\{g\in G:gA=A\}
$
is a \(\gamma_i\)-open subgroup of \(G\). Since the action preserves the
circular order, every \(g\in G_A\) must fix the endpoint \(a^+\), and hence
must fix \(a\). Thus \(G_A\subseteq G_a\). Consequently \(G_a\) is
\(\gamma_i\)-open. Again Lemma~\ref{l:stabilizers}.4 gives
\(\gamma_i=\tau_0\). 

\end{proof}

\begin{corollary} \label{c:happy} 
	Let 
	$\s \in  \mathcal{T}_{\downarrow}(G,\t_0)$ such that 
	$\s/H=\t_{0}/H.$ Then $\s=\t_0$. 
\end{corollary}
\begin{proof}
	By the proof of Lemma~\ref{l:CosetTopologies}, there exists
	\(\gamma\in\Sigma_H(G)\) such that
	\[
	\gamma\subseteq\sigma,\qquad \gamma/H=\sigma/H.
	\]
Since \(\sigma/H=\tau_0/H\), we have $\s/H \neq \t_{\T}/H$. Hence, Lemma~\ref{l:equalities}.1 
implies that the corresponding
compactification is not the circle compactification \(f_1\). Hence, by Lemma~\ref{l:equalities}.2, \(\gamma=\tau_0\). Therefore
	\[
	\tau_0=\gamma\subseteq\sigma\subseteq\tau_0,
	\]
	and so \(\sigma=\tau_0\).

\end{proof} 

%
%
%
%

%
%
%

The following useful result is known as 
\textit{Merson's Lemma}. 	

\begin{fact} \label{f:Merson} (Merson's Lemma \cite[Lemma 7.2.3]{DPS89} or \cite[Lemma 4.4]{DM14}) 
	
	Let $(G,\gamma)$ be a (not necessarily Hausdorff) topological group and $H$ be a subgroup of $G.$ If 
	$\gamma_1\subseteq \gamma$ is a coarser group topology on $G$ 
	such that $\gamma_1|_{H}=\gamma|_{H}$ and $\gamma_1/H=\gamma/H$, then $\gamma_1=\gamma.$
\end{fact}

\begin{definition} \label{d:Mk} 
Let $\nu=[a_1, \cdots, a_n]$ (with $n \geq 3$) be a cycle in  $\Q_0$. 
For every $k \in \{1, \cdots, n\}$ modulo $n,$ 
define the following subset of $G$: 
$$
M_k:=\{g \in G: \ 
g(x)=x \ \ \ \forall \ x \notin (a_{k-1}, a_{k+1})\}.
$$
\end{definition}

Then $M_k(y)=\{g(y): g \in M_k\}=(a_{k-1}, a_{k+1})$ 
for every $y \in (a_{k-1}, a_{k+1})$ by 
the c-ultrahomogeneity of the action of $G$ on $\Q_0$ 
(compare Lemma~\ref{l:stabilizers}.7, with
\(M_k=M_{[a_{k-1},a_{k+1}]}\)).  


\begin{lemma} \label{l:u1} 
	Let $\s \in  \mathcal{T}_{\downarrow}(G,\t_0)$ with $\s \neq \t_0.$ Suppose that $U \in N_e(G,\s), U^{-1}=U$ and  $\nu=[a_1, \cdots, a_n]$ is a cycle 
	in $\Q_0$ such that $n > 7$ and 
		\begin{equation} \label{eq:u1} 
			G_{a_1, \cdots, a_n} \subseteq U.   
		\end{equation}
	Then $M_k \subseteq U^4$ for every $k \in \{1, \cdots, n\}$ (modulo $n$).  
\end{lemma}
\begin{proof}  Fix $k \in \{1, \cdots, n\}$. 
		Let $f \in M_k$. We have to show that $f \in U^4$.  
	One may assume that $f(a_k) \neq a_k$ (otherwise, $f \in G_{a_1, \cdots, a_n}$, hence, $f \in U$ by Equation \ref{eq:u1}).  
	Since $[a_{k-1}, a_k, a_{k+1}]$ and $f \in M_k \subset  \Aut(\Q_0),$ 
	we have 
	$$[f(a_{k-1}), f(a_k), f(a_{k+1})]=[a_{k-1}, f(a_k), a_{k+1}].$$
	 That is, $f(a_k) \in (a_{k-1},a_{k+1})$.  
	Without loss of generality, assume that $f(a_k) \in  (a_{k},a_{k+1})$ (the case of $f(a_k) \in (a_{k-1},a_{k})$ is very similar). 
\sk 	
	
	Since $n>7$, we can assume below that the elements $a_{k-3}, a_{k-2}, a_{k-1}, a_k, a_{k+1}, a_{k+2}, a_{k+3}$ are different. 
	By Lemma \ref{l:minimum}, the compact-open topology $\t_{\T}$ 
	(coming from $\T$) is the least element in $\mathcal{T}_{\downarrow}(G,\t_0)$. Therefore, $\t_{\T} \subseteq \s$. There exists $V \in N_e(G,\t_{\T})$ such that $V^{-1}=V$ and 
	\begin{equation} \label{eq:great} 
	v(a_i)\in (a_{i-1},a_{i+1})
	\quad\text{and}\quad
	v(f(a_k))\in (a_k,a_{k+1}) 
	\end{equation}
	for every $v \in V$ and every $i  \in \{1, \cdots, n\}$ (cyclically modulo $n$). 
	Clearly, $G_{a_1, \cdots, a_n} \subseteq V$.
	Define $W:=U \cap V$. Then still $W \in N_e(G,\s)$ and also  
	\begin{equation} \label{eq:u3} 
		G_{a_1, \cdots, a_n} \subseteq W.   
	\end{equation}
	
	Choose $g \in W$ such that $g(a_k) \neq a_k$. 
	Such an element exists; otherwise, we would have $W \subseteq G_{a_k}$, which would imply $\sigma=\tau_0$ by Lemma \ref{l:stabilizers}.  
	 By our choice of $W \subseteq V$, we have  
	$g(a_i) \in (a_{i-1},a_{i+1})$ for every $i$ modulo $n$. In particular, $g(a_{k-1}) \in (a_{k-2},a_{k})$ and $g(a_k) \in (a_{k-1},a_{k+1})$. Observe also that since $g \in \Aut(\Q_0)$, it preserves the cyclic order. Hence, $g(\nu)=[g(a_1), \cdots, g(a_n)]$ is also a cycle in $\Q_0$. 
	Since $a_k \neq g(a_k) \in (a_{k-1},a_{k+1})$, there are two possibilities: $g(a_k) \in (a_{k-1},a_{k})$ or $g(a_k) \in (a_{k},a_{k+1})$. 
	We can assume (up to replacing $g$ by $g^{-1}$) that $g(a_k) \in (a_{k-1},a_{k})$ 
	(Indeed, if 
	$g(a_k) \in (a_{k},a_{k+1})$ then $[g(a_{k-1}),a_k, g(a_k)]$. 
	Therefore, $[a_{k-1}, g^{-1}(a_k), a_k]$ and $g^{-1} \in W$).

	
\sk 
	 We have the following cycle  
	 $[a_{k-1}, g(a_{k}), a_k, f(a_k), a_{k+1}]$. By multiplying on $g^{-1},$ we obtain the cycle $[g^{-1}(a_{k-1}), a_{k}, g^{-1}(a_{k}), g^{-1}(f(a_k)), g^{-1}(a_{k+1})]$.
	 
	 By Equation \ref{eq:great}, $g^{-1}(a_{k-1}) \in (a_{k-2},a_k)$.  
	 Define $s \in \Q_0$ as the maximum between $g^{-1}(a_{k-1})$ and $a_{k-1}$ in the natural linear order of $[a_{k-2},a_k]$ where $a_{k-2}$ is the least element (see Remark \ref{r:cech}.2). 
	  In other words, $s=g^{-1}(a_{k-1})$ if $g^{-1}(a_{k-1}) \in [a_{k-1},a_k)$ and $s=a_{k-1}$ if $g^{-1}(a_{k-1}) \in (a_{k-2},a_{k-1}]$. 
	 
	 By \ref{eq:great}, we have $g^{-1}(f(a_k)) \in (a_k,a_{k+1})$. This implies that $a_{k+1} \in  (g^{-1}(f(a_k)), a_{k+2})$. 
	 Again, by \ref{eq:great}, $g(a_{k+2}) \in (a_{k+1},a_{k+3})$. It follows that $[f(a_k), a_{k+1}, g(a_{k+2})].$ Thus, 
	 $[g^{-1}(f(a_k)), g^{-1}(a_{k+1}), a_{k+2}].$
	 Therefore,  $g^{-1}(a_{k+1})$ belongs to the same interval $(g^{-1}(f(a_k)), a_{k+2})$ as does $a_{k+1}$. 
	 
	 Define $t \in \Q_0$ as the minimum between $g^{-1}(a_{k+1})$ and $a_{k+1}$ in the natural linear order of the interval $(g^{-1}(f(a_{k})),a_{k+2})$. In other words, $t=g^{-1}(a_{k+1})$ if $a_{k+1} \in [g^{-1}(a_{k+1}),a_{k+2})$ and 
	 $t=a_{k+1}$ if $g^{-1}(a_{k+1}) \in (a_{k+1}, a_{k+2}]$. Note that 
	 \begin{equation} \label{eq:notin} 
	 	a_i \notin (s,t), \ \ \ g^{-1}(a_i) \notin (s,t) \ \ \ \ \forall  i \neq k 
	 \end{equation} 
	 and, on the other hand, by the construction, $g^{-1}(f(a_k))$ and $g^{-1}(a_k)$ both belong to the interval $(s,t)$. 
	 By c-ultrahomogeneity (Lemma \ref{l:stabilizers}) one may   
	 choose $\varphi \in G$ such that 
	 \[
	 \varphi\bigl(g^{-1}(f(a_k))\bigr)=g^{-1}(a_k),
	 \qquad
	 \varphi(a_k)=a_k,
	 \]
and $\varphi(x)=x$ for every $x \notin (s, t)$. Therefore, $\varphi(a_i)=a_i$ for every $i \in \{1, \cdots, n\}$. Then, $\varphi \in G_{a_1, \cdots, a_n}$.

	  Clearly, 
	 $
	 (g \varphi g^{-1}f)(a_k)=a_k.
	 $ 
	 Using Equation \ref{eq:notin} we have 
	 $(g \varphi g^{-1}f)(a_i)=a_i$ for every $i \neq k$.  This means that $g \varphi g^{-1}f \in G_{a_1, \cdots, a_n} \subseteq  W$,    
	 which implies that   
	 \[
	 f\in g\varphi^{-1}g^{-1}G_{a_1,\ldots,a_n}
	 \subseteq W\cdot G_{a_1,\ldots,a_n}\cdot W\cdot G_{a_1,\ldots,a_n}
	 \subseteq W^4\subseteq U^4.
	 \]
	 
\end{proof}

\begin{lemma} \label{l:NAwithCO}  
	It is impossible to find $\s \in  \mathcal{T}_{\downarrow}(G,\t_0)$ such that 
	$$
	\s|H=\t_0|H \ \wedge \ \s/H=\t_{\T}/H.
	$$
\end{lemma}
\begin{proof} As before, $H:=G_{a_0}$, where $a_0 \in \Q_0$.  Choose any $b \neq a_0$ in $\Q_0$. 
	Since $\s|H=\t_0|H$ and $G_{a_0, b} \in N_e(G,\t_0)$,  there exists $U \in N_e(G,\s)$ such that $U=U^{-1}$ and 
		\begin{equation} \label{eq:H} 
	U^5 \cap H \subseteq G_{b} \cap H.
	\end{equation}  
	Since \(\sigma/H=\tau_{\mathbb T}/H\neq \tau_0/H\), we have
	\(\sigma\neq\tau_0\). Since \(\sigma\subseteq\tau_0\), the neighbourhood
	\(U\) is \(\tau_0\)-open. Hence there exists a sufficiently long cycle
	\(\nu=[a_1,\ldots,a_n]\) in \(\mathbb Q_0\) such that
	\[
	G_{a_1,\ldots,a_n}\subseteq U.
	\]
	By refining the cycle if necessary, we may suppose that $n>7$, 
	\[
	a_0\in (a_n,a_1)
	\quad\text{and}\quad
	b\notin (a_n,a_1).
	\]
	
	Again, by Lemma~\ref{l:minimum}, \(\tau_{\mathbb T}\subseteq\sigma\).
	Choose a symmetric \(V\in N_e(G,\tau_{\mathbb T})\) such that
	\[
	v(a_i)\in(a_{i-1},a_{i+1})\quad (i=1,\ldots,n)
	\]
	and
	\[
	v(a_0)\in(a_n,a_1)
	\]
	for every \(v\in V\). Clearly, \(G_{a_1,\ldots,a_n}\subseteq V\). Put
	\(W:=U\cap V\). Then \(W\in N_e(G,\sigma)\), \(W=W^{-1}\), and 
		\begin{equation} \label{eq:H2} 
		W^5 \cap H \subseteq G_{b} \cap H \ \ \ \wedge \ \ \ 
		G_{a_1, \cdots, a_n} \subseteq W. 
	\end{equation}

	Since \(\sigma\neq\tau_0\), Lemma~\ref{l:stabilizers}.4 implies that
	\(G_b\notin\sigma\). Therefore \(W\) is not contained in \(G_b\). Choose
	\(w\in W\) such that \(w(b)\neq b\). Then \(w(a_0)\neq a_0\). Indeed, if
	\(w(a_0)=a_0\), then 
	$$w\in W\cap H\subseteq U^5\cap H\subseteq G_b\cap H,$$ 
	contrary to \(w(b)\neq b\). 
	Define
	\[
	M_0:=\{g\in G:g(x)=x \ \forall x\notin(a_n,a_1)\}.
	\]
	By the choice of \(V\), we have \(w(a_0)\in(a_n,a_1)\). Hence, by circular
	ultrahomogeneity inside the interval \((a_n,a_1)\), there exists
	\(m\in M_0\) such that \(m(w(a_0))=a_0\). Since \(b\notin(a_n,a_1)\), we have
	\[
	(mw)(b)=w(b)\neq b.
	\]
	Thus \(mw\in H\) and \(mw\notin G_b\).
	
	Finally, \(M_0\subseteq M_1\), where \(M_1\) is the subgroup from
	Definition~\ref{d:Mk} corresponding to the cycle \(\nu\). Therefore
	Lemma~\ref{l:u1} gives
	\[
	M_0\subseteq M_1\subseteq W^4.
	\] 
Hence, $mw \in W^5 \cap H$. On the other hand, 
$mw \notin G_{b} \cap H$. This contradiction (compare Equation \ref{eq:H2}) completes the proof.  
  
\end{proof}

\begin{lemma} \label{l:BOTHco} 
	Let $\s_1, \s_2 \in \mathcal{T}_{\downarrow}(G,\t_0)$ such that 
	$$
	\s_1|H=\s_2|H \ \wedge \ \s_1/H=\s_2/H=\t_{\T}/H.
	$$
	Then $\s_1=\s_2$. 
\end{lemma}
\begin{proof} 
	Put \(\sigma:=\sup\{\sigma_1,\sigma_2\}\). Then
	\(\sigma\in T_{\downarrow}(G,\tau_0)\), \(\sigma_1\subseteq\sigma\),
	\(\sigma_2\subseteq\sigma\), and, since
	\(\sigma_1|_H=\sigma_2|_H\), Equation~\((2.1)\) gives
	\[
	\sigma|_H=\sigma_1|_H=\sigma_2|_H.
	\]
	Also,
	\[
	\tau_{\mathbb T}/H=\sigma_1/H=\sigma_2/H\subseteq \sigma/H.
	\]
	By Lemma~\ref{l:equalities}, there are only two possibilities.
	
	If \(\sigma/H=\tau_{\mathbb T}/H\), then Merson's Lemma \ref{f:Merson} applied to
	\(\sigma_1\subseteq\sigma\) gives \(\sigma_1=\sigma\). Similarly,
	\(\sigma_2=\sigma\). Hence \(\sigma_1=\sigma_2\).
	
	If \(\sigma/H=\tau_0/H\), then Corollary~\ref{c:happy} gives
	\(\sigma=\tau_0\). Hence
	\[
	\sigma_1|_H=\sigma|_H=\tau_0|_H.
	\]
	Together with \(\sigma_1/H=\tau_{\mathbb T}/H\), this contradicts
	Lemma~\ref{l:NAwithCO}. Therefore this second case cannot occur.
	
\end{proof}

\sk 
Now, we can complete the \textbf{proof of Theorem} \ref{t:counterexample}.
\begin{proof} By Lemma \ref{l:HisnotCo-min} 
 it is enough to check that $H$ is inj-key in $G$.  
Let $\s_1, \s_2 \in \mathcal{T}_{\downarrow}(G,\t_0)$ such that $\s_1|H=\s_2|H$. We have to show that $\s_1=\s_2$. 
By Lemma \ref{l:equalities} we have only two admissible coset topologies. Therefore, we have the following three possibilities:
\begin{enumerate}
	\item $\s_1 / H=\s_2/H=\t_0/H$.  
	\item $\s_1/H=\s_2/H=\t_{\T}/H$.  
	\item 
	 $\s_1/ H=\t_0/H \wedge \s_2/H=\t_{\T}/H$ 
\end{enumerate}

\sk 
For (1), use Corollary \ref{c:happy} which ensures that $\s_1=\s_2=\t_0$. 

For (2), use Lemma \ref{l:BOTHco} to conclude directly $\s_1=\s_2$. 
 
For (3),  
Corollary \ref{c:happy} guarantees that $\s_1=\t_0$. 
Then our assumption $\s_1|H=\s_2|H$ implies that $\s_2|H=\t_0|H$. However, Lemma \ref{l:NAwithCO} tells us that such $\s_2$ does not exist. 

\end{proof} 

\begin{question} \label{q:BPestov} 
	Finally, back to Pestov's Question \ref{q:Pest}, it would be interesting to present concrete examples (if they exist) of Polish groups with metrizable $M(G)=\widehat{G/H}$ such that the corresponding extremely amenable subgroup $H$ is not inj-key in $G$.  	
\end{question}

\sk 
\section{Other directions and questions} \label{s:Q} 

It seems attractive to look for additional natural examples, coming from different sources, which respond to Problem~\ref{p:main}, not necessarily to Question~\ref{q:Pest}.

\subsection*[The Polish group Aut(Q,<=)]{The Polish group \(G:=\Aut(\mathbb Q,\leq)\)}

The linearly ordered rational case suggests a closely related example. 
It is expected, and the same method should establish, that for every $a\in\mathbb{Q}$, the stabilizer $H:=G_a$ is inj-key but not co-minimal 
in the Polish group \(G:=\Aut(\mathbb Q,\leq)\). The relevant greatest \(G\)-compactification
\(\beta_G(G/G_a)\) is a linearly ordered metrizable \(G\)-flow; it can be
described as the compact interval \([0,1]\) in which every internal rational
point \(q\in\mathbb Q\cap(0,1)\) is replaced by an ordered triple
\(q^-,q,q^+\). See \cite[Remark 3.8]{Me-OrdSem} and also
\cite[Example 5.4(b)]{Me-MaxEqComp}.

	This $G$-compactification \(\beta_G(G/G_a)\) has six \(G\)-orbits: the middle rational orbit
	\(\mathbb Q\cap(0,1)\), the two twin rational orbits
	\((\mathbb Q\cap(0,1))^-\) and \((\mathbb Q\cap(0,1))^+\), the irrational
	orbit \(J=[0,1]\setminus\mathbb Q\), and the two endpoint orbits
	\(\{0\}\) and \(\{1\}\). 
	

\sk 
Let \(H\) be a closed co-compact subgroup of a Polish group \(G\). Then \(H\)
is Raikov complete, and by Fact~\ref{p:inj-key=co-min for central}.3 it is a key subgroup of \(G\).
In particular, this applies to point stabilizers in transitive actions of
Polish homeomorphism groups $G=\Homeo(K)$ on compact homogeneous spaces $K$, as in
Example~\ref{ex:Effros}. 
In view of Problem \ref{p:main} it is interesting to verify
whether 
$H$ might be even inj-key but not co-minimal. For such a subgroup to fail co-minimality, it is of course necessary that
the ambient group $\Homeo(K)$  is not minimal. 

\sk  
\subsection*{J. van Mill's result for Menger's continuum} 
By a result of J. van Mill \cite{Mill07}, the Polish homeomorphism group
\(\Homeo(K)\) of a homogeneous metrizable continuum \(K\) need not be minimal;
this answered a question of L. Stoyanov. In particular, one may take
\(K=\mu_n\), the \(n\)-dimensional Menger continuum, for every \(n\geq 1\).
 This leads to the following 

\begin{question} \label{q:Mill} 
Is the (key) subgroup $H:=\operatorname{St}(x_0)$ of $G:=\Homeo(\mu_n)$ inj-key but not co-minimal for the Menger continuum $\mu_n$? 
\end{question}

\sk 
\subsection*{The homeomorphism group of the 2-cell}

Another natural candidate is
$
G:=\Homeo([0,1]^2)
$
and the stabilizer \(H:=\operatorname{St}(p)\), where \(p \in\partial([0,1]^2)\).
The group \(G\) acts transitively on the boundary, and hence \(G/H\) identifies
with the boundary circle. Thus \(H\) is a key subgroup.

By results of D. Gamarnik \cite{Gam}, the Polish group
\(\Homeo([0,1]\times[0,1])\) is not minimal. Moreover, according to
\cite[page 10]{Gam}, there exists a strictly coarser Hausdorff group topology
on \(G\) in which \(H\) is dense. Therefore \(H\) is not co-minimal. 
It is therefore natural to ask whether this subgroup $H$ is inj-key in $G$.



%
%



\sk 
\subsection*{Epimorphic embeddings of topological groups}

It may also be useful to examine embeddings
$
i\colon H\hookrightarrow G
$
which are epimorphisms in the category of Hausdorff topological groups.
Every dense embedding is an epimorphism, but dense subgroups are always
co-minimal. On the other hand, Uspenskij proved \cite{Usp-epic} that
epimorphisms in the category of Hausdorff topological groups need not have
dense range.

This raises the question whether non-dense epimorphic embeddings can produce
new examples of key or inj-key subgroups which are not co-minimal.

		\bibliographystyle{plain}

\begin{thebibliography}{10}
			

		
		

		
	\bibitem{BMT} 
	I. Ben Yaacov, J. Melleray and T. Tsankov, 
\textit{Metrizable universal minimal flows of Polish groups have a comeagre orbit}, Geometric and Functional Analysis 
\textbf{27} (2017), 67--77 

		
\bibitem{BT} 
I. Ben Yaacov and T. Tsankov, 
	 \textit{Weakly almost periodic functions, model-theoretic stability, and minimality of topological groups}, Trans. Amer. Math. Soc. \textbf{368} (2016), no. 11,
	8267--8294 

\bibitem{BZ} 
G. Basso, A. Zucker,   
\textit{Topological groups with tractable minimal dynamics}, 2024, 
arxiv.org/abs/2412.05659  

\bibitem{Bourb} 
N. Bourbaki, \textit{General Topology}, v. 1, Springer, 1966

\bibitem{Cech}
E. \v{C}ech, \emph{Point Sets}, Academia, Prague, 1969 

\bibitem{CG} 
X. Chang and P. Gartside, 
\textit{Minimum topological group topologies}, 
J. Pure Appl. Algebra \textbf{221} (2017), 2010--2024  
		
		
\bibitem{DM10} D. Dikranjan and  M.  Megrelishvili, \emph{Relative minimality and co-minimality of subgroups in topological groups,}
		Topology Appl. \textbf{157} (2010), 62--76 
		
\bibitem{DM14}	D. Dikranjan and  M.  Megrelishvili, \emph{Minimality conditions in topological groups,}
in: Recent Progress in General Topology III, 229--327, K.P. Hart, J. van Mill, P. Simon (Eds.), Springer, Atlantis Press, 2014, pp.  229--327 
			
			
\bibitem{DPS89} D. Dikranjan, Iv. Prodanov and  L. Stoyanov, \emph{Topological Groups: Characters, Dualities and Minimal Group
Topologies,} Pure and Appl. Math. \textbf{130}, Marcel Dekker, New York-Basel, 1989 
				
	

\bibitem{Duch20} 
B. Duchesne, 
\textit{Topological properties of Wazewski dendrite groups}, 
J. Ec. Polytech. - Math. \textbf{7} (2020), 431-477


\bibitem{Gam} 
D. Gamarnik, Minimality of the group $\Aut(C)$, Serdica Math. J. \textbf{17}, n. 4 (1991) 197--201

\bibitem{Gl-book}
E. Glasner, \emph{Proximal flows,} Lecture Notes in Mathermatics,
\textbf{517}, Springer-Verlag, 1976

\bibitem{GM-CircOrd18}
E. Glasner and M. Megrelishvili, 
\emph{Circularly ordered dynamical systems}, Monatsch. Math.  \textbf{185} (2018), 415--441


\bibitem{GM-UltraHom}
E. Glasner and M. Megrelishvili, 
\textit{Circular orders, ultra-homogeneous order structures and
	their automorphism groups},  
AMS Contemporary Mathematics \textbf{772}, ``Topology, Geometry, and Dynamics: Rokhlin-100" (ed.: A.M. Vershik, V.M. Buchstaber, A.V. Malyutin) 2021, pp. 133--154.
Updated version ArXiv:1803.06583  

\bibitem{GM-TC} 
E. Glasner and M. Megrelishvili, 
\textit{Todor\'{c}evi\'{c}' Trichotomy and a hierarchy in the class of tame dynamical systems}, Trans. Amer. Math. Soc. \textbf{375} (2022), 4513--4548. ArXiv:2011.04376 

\bibitem{GW-02}
E. Glasner and B. Weiss, 
{\em Minimal actions of the group $S(\Z)$ of permutations of the integers},
Geometric and Functional Analysis {\bf 12} (2002), 964--988

\bibitem{GW-03}
E. Glasner and B. Weiss, 
{\em  The universal minimal system for the group of homeomorphisms of the Cantor set},
Fund. Math. {\bf 176},  (2003),  277--289 

\bibitem{IbMe} 
T. Ibarlucia and M. Megrelishvili, 
\textit{Maximal equivariant compactification of the Urysohn spaces and other metric structures,} 
Advances in Math. \textbf{380} (2021), 107599. ArXiv:2001.07228
		
\bibitem{KPT} 
A.S. Kechris, V.G. Pestov and S. Todorcevic, \emph{Fra\"{i}ss\'{e} limits, Ramsey theory, and topological dynamics of automorphism groups}, 
Geometric and Functional Analysis \textbf{15} (2005), 106--189 

\bibitem{KS}  
K.L. Kozlov and B.V. Sorin, \textit{Enveloping semigroups as compactifications of topological groups}, 2025, arXiv:2509.17577  

\bibitem{Kwi18} 		
A. Kwiatkowska, \textit{Universal minimal flows of generalized Ważewski dendrites}, J. Symb. Log. \textbf{83} (2018), 1618--1632




		\bibitem{MEG95}	M. Megrelishvili, \emph{Group representations and construction of minimal topological groups}, 
		Topology Appl. \textbf{62} (1995), 1--19  
	
	

	\bibitem{MEG04} 
	M. Megrelishvili, 
	\textit{Generalized Heisenberg groups and Shtern's question}, 
	Georgian Math. J.  \textbf{11:4} (2004), 775--782 
	
	
\bibitem{Me-opit07}
M. Megrelishvili,
\emph{Topological transformation groups: selected topics}. In: 
Open Problems in Topology II (Elliott Pearl, editor), Elsevier
Science, 2007, pp. 423--438 
	
%


\bibitem{Me-OrdSem} 
M. Megrelishvili,
\textit{Orderable groups and semigroup compactifications,} Monatsch. Math. \textbf{200} (2022), 903--932 

\bibitem{Me-MaxEqComp}  
M. Megrelishvili, 
\textit{Maximal equivariant compactifications}, Topology Appl. \textbf{329} (2023), 108372

\bibitem{Me-b}
M. Megrelishvili, 
\textit{Topological Group Actions and Banach Representations}, unpublished book, Available on author's homepage, 2025
	
\bibitem{Me-CircTop}  
M. Megrelishvili, 
\textit{Circular orders: Topology and continuous actions}, arXiv:2512.17314, 2025	

\bibitem{MS-Fermat}
M. Megrelishvili and M. Shlossberg, \textit{Minimality of topological matrix groups and Fermat primes}, Topology Appl. \textbf{322} (2022), 108272  
	
\bibitem{MeSh-key}
M. Megrelishvili and M. Shlossberg, \textit{Key subgroups in topological groups,} Forum Mathematicum \textbf{38} (2026), 73-92


\bibitem{MNT} 
J. Melleray, L. Nguen Van The and T. Tsankov, 
\textit{Polish groups with metrizable universal minimal flows}, 
Internat. Math. Res. Notices, \textbf{2016}, Issue 5,  (2016), 1285--1307  

\bibitem{Mill07} 
J. van Mill, \textit{Homeomorphism groups of homogeneous compacta need not be minimal}, Topology Appl. \textbf{159} (2012) 2506--2509

\bibitem{van-the} 
L. Nguyen van Th\'{e}, \emph{More on the Kechris-Pestov-Todorcevic correspondence: precompact expansions}, 
Fund. Math. \textbf{222} (2013), no. 1, 19--47

\bibitem{Pest98} V. Pestov, \emph{On free actions, minimal flows and a problem
	by Ellis}, Trans. Amer. Math. Soc. {\bf 350} (1998), 4149--4165 
	
	\bibitem{Pe-nbook}
	V. Pestov, \textit{Dynamics of infinite-dimensional groups. The Ramsey-Dvoretzky-Milman phenomenon,} University Lecture Series, v. {\bf 40}, AMS, Providence, 2006 

\bibitem{Pest-Smirnov}
V. Pestov, 
\textit{A topological transformation group without non-trivial equivariant compactifications}, 
Advances in Math. \textbf{311} (2017) 1--17



\bibitem{RD}
W. Roelcke and S. Dierolf, {\em Uniform Structures on Topological
	Groups and Their Quotients\/}, McGraw-Hill, 1981 


\bibitem{Sorin25-Sbornik} 
G.B. Sorin, Semigroup compactifications of groups of automorphisms of ultra-homogeneous cyclically ordered sets, Math. Sbornik (in Russian), 217:2 (2026), 154--179  


		
	

\bibitem{Usp-epic} V.V. Uspenskij, \emph{The epimorphism problem for Hausdorff
	topological groups}, Topology Appl. \textbf{57} (1994), 287--294

\bibitem{Vr-can75} J. de Vries, \emph{Can every Tychonoff $G$-space equivariantly be embedded in a compact Hausdorff $G$-space?}, Math. Centrum 36,
Amsterdam, Afd. Zuiver Wisk., 1975

\bibitem{Vr-Embed77} J. de Vries, \emph{Equivariant embeddings of
	$G$-spaces}, in: J. Novak (ed.), General Topology and its
Relations to Modern Analysis and Algebra IV, Part B, Prague, 1977,
485--493

\bibitem{Zucker14} 
A. Zucker, 
\textit{Topological dynamics of automorphism groups, ultrafilter combinatorics, and the generic point problem.}  
Trans. Amer. Math. Soc. \textbf{368} (2016),
6715--6740 

	
	\end{thebibliography}


\end{document}